\documentclass[12pt]{article}
\usepackage[latin1]{inputenc}
\usepackage{amscd}
\usepackage{amsmath,amsthm,amssymb,amsfonts}
\usepackage{mathrsfs}
\usepackage{makeidx}
\usepackage{graphicx}
\usepackage{authblk}
\usepackage{tikz,tkz-tab}
\usepackage{color}
\usepackage{amsmath}
\usepackage{amsfonts}
\usepackage{amssymb}
\usepackage{geometry}
\makeatletter

\@addtoreset{equation}{section}
\makeatother

\newtheorem{theorem}{Theorem}[section]
\newtheorem{proposition}{Proposition}[section]
\newtheorem{lemma}{Lemma}[section]

\def\R{\mathbb{R}}

\def \p{\partial}

\def \e {\varepsilon}
\def \I {\mathcal I}

\def \J {\mathcal J}

\def \a {\alpha}
\def \k {\kappa}

\def \d {\delta}
\def \v {\varphi}
\def \b { \beta}
\def \l {\tilde{\kappa}}
\def \g {\gamma}
\def\ep{\varepsilon}
\def\eps{\varepsilon}

\def \dd{\sqrt{\delta}}

\begin{document}
\title{One dimensional phase transition problem modelling  striped spin orbit coupled Bose-Einstein condensates}
\author{Amandine Aftalion\footnote{Ecole des Hautes Etudes en Sciences Sociales, PSL Research University, CNRS UMR 8557, Centre d'Analyse et de Math\'ematique Sociales, 54 Boulevard Raspail, 75006 Paris, France}
\and  R\'emy Rodiac\footnote{Universit\'e Catholique de Louvain, Institut de Recherche en Math\'ematique et Physique, Chemin du Cyclotron 2 bte L7.01.01, 1348 Louvain-la-Neuve, Belgium}}
\date{\today}

\maketitle
\begin{abstract} We study the behaviour of a Modica-Mortola phase transition type problem with a non-homogeneous Neumann boundary condition. According to the parameters of the problem, this leads  to the existence of either one component occupying most of the condensate with an outer boundary layer containing the other component, or to many interfaces, on a periodic pattern. This is related to the striped behaviour of a two component Bose-Einstein condensate with spin orbit coupling in one dimension. We prove that minimizers of the full Gross-Pitaevskii energy in 1D converge, in the Thomas-Fermi limit of strong intra-component interaction, to those of the simplified Modica-Mortola problem we have studied in the first part.
\end{abstract}

\section{Introduction}
The aim of this paper is to study a one dimensional  functional which models vortex stripes in two component condensates, namely
\begin{eqnarray}\label{Energyintro}
G_{\e,\d,\k}(v,\varphi)&=&\frac{1}{2}\int_{0}^1v'(x)^2\ dx+\frac{1}{4\e^2}\int_{0}^1(1-v^2(x))^2\ dx
+ \frac{1}{8}\int_{0}^1v^2(x)\varphi'(x)^2\ dx \nonumber \\  &+&\frac{\d}{8\e^2}\int_{0}^1v^4(x)\sin^2\v(x)\ dx -\frac{\k}{2}\int_{0}^1v^2(x)\varphi'(x)\ dx.
\end{eqnarray} The function $v^2$ describes  the total density of the two components, while the value of $\sin \varphi$ allows to discriminate between component 1 and component 2.
Here $\e$ is a small positive parameter which describes the interactions inside each component, $\d$ is a positive parameter  which describes the interactions between the two components and $\k$ is the spin orbit modulation which we assume to be positive. We impose  the constraint $\int_{0}^1v^2(x)dx =1$, which means that the total mass is prescribed, and set $\varphi (0)=0$.

We are going to study the limit of minimizers as $\ep$ tends to 0 according to the values of $\delta$ and $\kappa$. We will always work in a regime where $v^2$ tends to 1, and we will prove that this occurs if $\k$ is bounded, or if $\k$ blows up like $\sqrt\delta/\e$, and $\delta$ goes to zero as $\e$ tends to zero. Under this hypothesis, we will show that the behaviour of $\varphi$ depends on $\kappa$ and a reduced parameter
 \begin{equation}\label{defbeta}
 \b=\frac{ \e}{\sqrt{\d}},
 \end{equation}
and  is determined by the auxiliary problem
\begin{equation}\label{eqF}
 F_{\b, \k}(\v):=G_{\e,\d,\k}(1,\v)= \frac{1}{8} \int_0^1 \left(\v'(x)^2 +\frac 1 {\b^2} \sin^2 \v(x)\right)\ dx -\frac{\k}{2} \int_0^1 \v'(x)\ dx
 \end{equation} with $\varphi (0)=0$. We will focus on the case where $\b$ is small.
 The study of the energy $F_{\b, \k}$ is a Modica-Mortola type problem \cite{Modica1987} except for the non-homogeneous Neumann boundary condition at 1 which comes out in the Euler-Lagrange equations due to the term $\frac{\k}{2} \int_0^1 \v'(x)\ dx$. This Neumann boundary condition can lead to oscillations in $\varphi$ and thus stripes in the original problem. Namely,
 for $\k\beta<1/\pi$, that is $\kappa \ep/\sqrt{\delta} <1/\pi$, we will show that $\varphi$ converges exponentially fast to 0, except at the point 1, which means that the domain is almost nearly occupied by the same component;  if $\k\beta>1/\pi$, that is $\kappa \ep/\sqrt{\delta} >1/\pi$ (with $\kappa \ep/\sqrt{\delta}$ independent of $\e$), then $\varphi$ goes from 0 to $N\pi$ where $N$ is large, with a quasi periodic behaviour corresponding to many stripes. For $\kappa \ep/\sqrt{\delta} $ close to $1/\pi$, $\varphi$ stays between 0 and $\pi$, corresponding to the usual phase transition solution.

The paper is dedicated to the study of $F_{\b, \k}$ and the convergence of minimizers of $G_{\e,\d,\k}(v,\varphi)$ as $\ep$ tends to 0.
\subsection{Physical and mathematical motivation}
 Our motivation stems from  the new physics emerging in spin orbit coupled Bose-Einstein condensates, and in particular the existence of vortex stripes \cite{AMspin,hu2012,MLS,sinha,WGJZ,zhai}. Bose-Einstein condensates are quantum gases described by a complex valued wave function whose modulus is the density of atoms and whose phase is related to the singularities. Two component condensates are described by two wave functions and correspond to
   a single isotope in two different
hyperfine spin states, two different isotopes of the same atom or
isotopes of two different atoms.  According to the respective values of the inter-component and intra-component interactions, the minimizers exhibit very different properties in terms
 of shape of the bulk, defects and coexistence of the components or spatial separation \cite{MArot}. It turns out that the sign of the parameter
  $\delta$  plays an important  role: if $\delta<0$, the two components coexist
   while if $\delta>0$, they separate or segregate. The segregation behavior in two component condensates has been  studied by many authors: regularity of the wave function \cite{NoTaTeVe}, regularity of the interface \cite{CaffLin2}, asymptotic behavior near the interface \cite{AfSo,ABCP,BeLinWeiZhao,BeTer,dwz,Sourdis2016}, $\Gamma$-convergence to a perimeter type functional \cite{AftalionLetelier2015,GM,GoldmanLetelier2015} in the case of a trapped condensate. The coexisting case has been analyzed in \cite{ANS,Ga}.

 The interplay between the spin orbit coupling and the interaction parameter leads to very different and new patterns \cite{AMspin,hu2012,zhai}. In the case of strong repulsive interaction, the spin orbit coupled BEC minimizes the energy by spontaneously breaking the rotational symmetry, developing a spin modulation in an arbitrary direction, leading to one dimensional stripes \cite{MLS,sinha,WGJZ}, which is the main interest of our paper.

 When a two component condensate is spin orbit coupled, it is described by two wave functions $u_1$ and $u_2$ minimizing an energy depending on 3 parameters: $\kappa$ the spin modulation, $\delta$ which measures the interaction between the components, and $\eps$ a small parameter related to the intra-component self interaction. Instead of considering the harmonic trapping potential as in the experiments, we can assume that the system  is localized in a bounded two-dimensional domain $D$. Then the energy is:
 \begin{eqnarray}
E(u_1,u_2)&=& \frac{1}{2}\int_D |\p_xu_1+\kappa u_2|^2 +|\p_yu_1+i\kappa u_2|^2
 + |\p_x u_1-i\kappa u_1|^2+|\p_y u_2+\kappa u_1|^2 \nonumber \\
 &+ &\frac{1}{4\e^2}\int_D(|u_1|^2+|u_2|^2-1)^2+\frac{\delta}{2\e^2}\int_D |u_1|^2|u_2|^2
\label{enerful}\end{eqnarray} under the condition $\int_D |u_1|^2+|u_2|^2=1$. This energy is the same as in \cite{AMspin} up to a constant. When $\kappa=0$, the ground state depends on the sign of $\delta$. The case $\delta>0$ is known as the segregation case and the limiting behaviour is the minimal perimeter of the interface between the two components. In the case of fixed prescribed mass for each component, the limiting problem depends on $\delta$, either tending to 0, $+\infty$ or fixed, \cite{AftalionLetelier2015,GM,GoldmanLetelier2015}. More precisely, when $\delta$ tends to 0 then $v^2=u_1^2+u_2^2$ tends to 1 everywhere \cite{GM}. In the case of strong segregation ($\delta\to\infty$), then $v$ tends to 0 at the interface leading to a sharp interface \cite{AftalionLetelier2015}. On the other hand, if $\delta$ is fixed, then $v$ stays between zero and 1 \cite{GoldmanLetelier2015}. We will see that when the  $\k$ term is added to the problem, then for bounded $\k$ at least, $v$ tends to 1 strongly even at the interface.

 Since in this paper we  assume that the sum of the $L^2$ norms is fixed, instead of prescribed $L^2$ norm in each component, then the optimal solution for $\k=0$ is to have only one component with all the mass, that is no interface. When $\kappa$ is added to the problem and we impose a condition on the sum of the masses, the behaviour changes. It becomes energetically favorable to have an interface. For low $\kappa$, the numerical simulations \cite{AMspin} indicate that radial symmetry is preserved and the interfaces are circles, that is the components are made up of concentric annuli with a central disk, while for large $\kappa$, there is a breaking of symmetry leading to one dimensional stripes.
 In \cite{WGJZ}, this is described as standing waves, that is the two wave functions are in $\cos$ and $\sin$. Therefore, in this paper, as a first understanding of this phenomenon we reduce the energy (\ref{enerful}) to a one dimensional energy. More precisely we take $D=[0,1]$,  $u_1$ and $u_2$ real valued and we set $v^2=u_1^2+u_2^2$, $u_1= v \cos \varphi /2$, $u_2= v\sin \varphi /2$. We point out that $\varphi=0$ corresponds to component 1, while $\varphi=\pi/2$ to component 2.
  This change of functions turns the energy (\ref{enerful}) into (\ref{Energyintro}), up to the addition of a constant term.

  If $\k$ is fixed, we will show that $v^2$ tends to 1 and the domain is almost occupied by component 1, except on a thin layer at the interface. The fact that component 1 is privileged is due to our specific choice of boundary condition $\varphi (0)=0$. If $\k$ is of order $\sqrt\delta/\e$, and $\delta$ tends to 0, then we will show that $v^2$ tends to 1 everywhere, and the number of transitions (or pieces of each component) depends on an auxiliary problem for $\varphi$. If $\delta$ does not tend to zero, then $v$ does not tend to 1 at the interface, and the effect of the spin orbit takes place for even larger $\kappa$ and remains an open question, since the problem is no longer decoupled between $v$ and $\varphi$.
 \subsection{Main results}
For $v$ in $H^1((0,1))$ we set
\begin{equation}\label{Singularset}
\mathcal{S}(v):=\{x\in [0,1]; v(x)=0\}
\end{equation}
and
\begin{equation}
\I:=\{(v,\v)\in H^1((0,1))\times \left( H^1((0,1))\setminus \mathcal{S}(v)\right);  \ \v(0)=0 \text{ and } \int_{0}^1v^2(x)dx =1\}.
\end{equation}
We want to minimize in the space $\mathcal{I}$ the functional  $G_{\e,\d,\k}$ defined in \eqref{Energyintro}. We set $v^2\v'^2=0$, $v^2\v'=0$ and $v^4\sin^2\varphi=0$ on $\mathcal{S}(v)$. The main results obtained for $G_{\e,\d,\k}$ will be deduced from the study of an auxiliary problem which is also of independent interest. We recall that $\b$ is given by \eqref{defbeta} and we study the functional $F_{\b, \k}$ given by \eqref{eqF} defined in
\begin{equation}\label{defJ}
\mathcal{J}=\{ \v \in H^1((0,1)); \v(0)=0 \}.
\end{equation}
We also set
\begin{equation}\label{tildekappa}
\tilde{\k}:= \k \b
\end{equation} when necessary in order to study the case of  unbounded $\k$ as $\b$ goes to zero, that is $\k$ can depend on $\b$, though we do not write explicitly the dependence.  We want to find an expansion of the energy of a minimizer of $F_{\b,\k}$ and to describe the asymptotic behavior of these minimizers as $\b$ goes to zero. Note that this problem is a Modica-Mortola type problem with non-homogeneous Neumann boundary condition. The asymptotic behavior of the minimizers depends strongly on this Neumann condition, that is on the value of the parameter $\k$ which we allow to depend on $\b$. More precisely, we prove:
\begin{theorem}\label{caskinf1/pi}
Let $\k < \frac{1}{\pi \b}$, and $\l=\k \b$. There exists a unique minimizer $\v_\b$ of $F_{\b,\k}$ in $\mathcal{J}$.
 Moreover, $0<\v_\b(x)<\frac{\pi}{2}$ for all $x$ in $(0,1)$, $\v_\b \rightarrow 0$ in $\mathcal{C}^\infty(K)$ as $\b \rightarrow 0$ for every compact set $K \subset [0,1)$ and
\begin{equation}\label{eqdecroissanceexp}
\v_\b(x) < 2 \arctan \left[ (\tan \frac{\v_\b(1)}{2})e^{\frac{x-1}{\b}} \right].
\end{equation}
\begin{itemize}
\item[i)] If $\k \b=\l$ is independent of $\b$, then for $\b>0$ small enough, $\v_\b(1)=\arcsin(2\l)+o_\b(1)$
 and
\begin{equation}\label{DL11}
F_{\b,\k}(\v_\b)= \frac{1-\sqrt{1-4\l^2}}{4\b}-\frac{\l}{2\b}\arcsin(2\l)+o_\b(\b^n), \forall \ n \in \mathbb{N}^*.
\end{equation}Besides, if we set $\psi_\b(x):=\v_\b(1-\b x)$, then  $\psi_\b \rightarrow \psi_0$ in $\mathcal{C}^\infty_{\text{loc}}([0,+\infty))$ where $\psi_0(x)=2\arctan \left[ \tan \left(\frac{\arcsin(2\l)}{2} \right)e^{-x}\right]$.
\item[ii)] If $\k\b =o_\b(1)$ then $\v_\b(1)=2\k \b(1+o_\b(1))$ and
\begin{equation}\label{DL22}
F_{\b,\k}(\v_\b)=\frac{-\k^2 \b}{2}(1+o_\b(1)).
\end{equation}
Moreover, if we set $\Phi_\b(x):=\frac{\v_\b(1-\b x)}{2\k\b}$, then $\Phi_\b \rightarrow e^{-x}$ in $\mathcal{C}^\infty_{\text{loc}}([0,+\infty))$.
\end{itemize}
\end{theorem}
In this case, the Neumann boundary condition is too small to create a phase transition; component 1 occupies almost all the condensate, and $\v_\b$ goes to 0. The size of the boundary layer is of order $\beta$. It is only for $\tilde \kappa$ of order 1 that $\v_\b$ reaches a non zero value at 1 in the limit. The proof relies on the classical Modica-Mortola technique using the solution of $\psi'=\sin \psi$, with the Neumann boundary condition at 1 $\psi'(1)=2\kappa$, and the value at 1 which comes from the minimization of the energy, that is $\psi(1)=\arcsin (2\l)$. This solution  is exactly the function $\psi_0$ of the Theorem.
%

When $\k>\frac{1}{\pi \b}$, if $\l=\k \b$ is independent of $\b$ we observe a complete change of shape of a minimizer of $F_{\b,\k}$ in $\mathcal{J}$. Indeed such a minimizer satisfies that $\v_\b(1) \geq N \pi$ where $N$ is an integer of order $\frac{1}{\b}$. This means that there are many interfaces and all  phase transitions are of the same size because of the periodicity of the solution. More precisely, we have:
\begin{theorem}\label{casksup1/pi}
Let $\k >\frac{1}{\pi \b}$ and $\l= \k \b$ be independent of $\b$. Let $\v_\b$ be a minimizer of $F_{\b,\k}$ in $\mathcal{J}$. There exists a unique $T$ in $(0,1)$ such that $\v_\b(T)=\pi$. Besides,  $\v_\b$ is quasi-periodic in the following sense:
\begin{equation}\label{periodicityty}
\v_\b(x+T)=\pi+\v_\b(x)  \text{ for every } x \text{ in } [0,1-T].
\end{equation}
We set $N:= E(\frac{1}{T})$. There exist $0<c<C$, independent of $\b$, such that $\frac{c}{\b}\leq N \leq \frac{C}{\b}$. Furthermore, there exists a unique ${\tilde{\a}}_0$, with $0< {\tilde{\a}}_0 \leq 2\l$, defined by
\begin{equation}\label{defalpha0}
\int_0^\frac{\pi}{2} \sqrt{{\tilde{\a}}_0^2+\sin^2y}\ dy= \l \pi,
\end{equation} and we have
\begin{equation}\label{DLenergie2}
F_{\b,\k}(\v_\b)= \frac{-{\tilde{\a}}_0^2}{8\b^2}+O_\b\left(\frac{1}{\b}\right).
\end{equation} Let $\tilde{\v}_\b(x):= \v_\b(\b x)$ defined in $[0,\frac{1}{\b}]$.  Then, $\tilde{\v}_\b$ converges in $C^\infty(K)$ for every $K \subset [0,+\infty)$ to $\v_0$ the solution of
    \begin{equation}\label{eqlimitk>1/pi}
    \left\{
    \begin{array}{lcll}
    \v_0''&=&\sin \v_0 \cos \v_0 \text{ in } \R^+, \\
    \v_0(0)&=&0, \\
    \v_0'(0)&=& {\tilde{\a}}_0,
    \end{array}
    \right.
    \end{equation}
with ${\tilde{\a}}_0$ defined by \eqref{defalpha0}.
\end{theorem} The number $N$ is called the number of periods and $\v_\b(NT)= N \pi$. The proof relies again on an upper bound and lower bound, but taking into account the periodic solution of (\ref{eqlimitk>1/pi}). We will also see that when $\l$ gets large, the solution $\varphi_0$ becomes almost linear. The limiting case $1/\pi$ will be analysed in Proposition \ref{prop1surpi}.

Once we have obtained these results about the auxiliary problem we use them to describe the original problem. Though we do not write down the $\e$ dependence as $\k_\e$ and $\d_\e$, we allow $\k$ and $\d$ to depend on $\e$. We need the hypothesis $\e^2=o_\e(\d)$ to ensure that $\b=\e/\sqrt \d$ is small. When additionally, $\k$ blows up like $1/\e$, we also need to assume that $\d$ is small to ensure that $v^2$ tends to 1.

\begin{theorem}\label{minimizerspbcomplet}
Let $(v_\e,\v_\e)$ be a minimizer of $G_{\e,\d,\k}$ in $\mathcal{I}$, then:
\begin{itemize}
\item[a)] For $\k$ bounded and $\d$ fixed, then for $\e>0$ small, there exists $C>0$, independent of all the parameters, such that \begin{equation}\label{93}
\|v_\e-1\|_{L^\infty((0,1))}=o_\e(\e),\end{equation}
\begin{equation}\label{15}
\v_\e(1)=2 \k {\frac{\e}{\sqrt{\d}}}(1+o_\e(1)), \ \| \v_\e\|_{L^\infty((0,1))}\leq C \k {\frac{\e}{\sqrt{\d}}} ,
\end{equation}\begin{equation}
G_{\e,\d,\k}(v_\e,\v_\e)=\frac{-\k^2 \e}{2\sqrt{\d}}(1+o_\e(\e)).
\end{equation}
If we set $\Phi_\e(x)=\frac{\dd}{2\k \e} \v_\e(1-\frac{\e x}{\dd})$ then $\Phi_\e \rightarrow e^{-x}$ in $\mathcal{C}^1_{\text{loc}}([0,+\infty))$.

\item[b)] For $\k< \frac{\sqrt{\d}}{\pi \e}$, we set $\k=\frac{\l \sqrt{\d}}{\e}$. If $\l$ is independent of $\e$, $\d=o_\e(\e)$ and $\e^2=o_\e(\d)$ then \begin{equation}\label{94}
\|1-v_\e\|_{L^\infty((0,1))}=o_\e(\d^{1/4}), \ \  \v_\e(1)=\arcsin(2\l)(1+o_\e(1)),
\end{equation}
    \begin{equation}\label{eqG}
    G_{\e,\d,\k}(v_\e,\v_\e)=\frac{\sqrt{\d}}{\e}\left( \frac{(1-\sqrt{1-4\l^2})}{4}-\frac{\l}{2}\arcsin(2\l)\right)(1+o_\e(1)).
    \end{equation} Moreover, $\| \v_\e\|_{L^\infty((0,1))}\leq \arcsin(2\l)(1+o_\e(1))$.
We set $\psi_\e(x)=\v_\e(1-\frac{\e x}{\dd})$ and we have: $\psi_\e \rightarrow\psi_0$ in $\mathcal{C}^1_{\text{loc}}([0,+\infty))$ with $\psi_0(x)=2\arctan \left[ \tan \left(\frac{\arcsin(2\l)}{2} \right)e^{-x}\right]$.

\item[c)] For $\k>\frac{\sqrt{\d}}{\pi \e}$, we set $\k=\frac{\l \sqrt{\d}}{\e}$. If $\l$ independent of $\e$, $\e^2=o_\e(\d)$ and $\d=o_\e(\e)$ then $\|v-1\|_{L^\infty((0,1))} \leq C \sqrt{\frac{\d}{\e}}$ and
\begin{equation}
G_{\e,\d,\k}(v_\e,\v_\e)=\frac{-\tilde{\alpha}_0^2 \d}{{8}\e^2}(1+o_\e(1)),
\end{equation} where $\tilde{\alpha}_0$ is defined by \eqref{defalpha0}.
We let $\tilde{\v}_\e(x):=\v(\frac{\e x}{\dd})$ defined in $[0,\frac{\dd}{\e}]$. If $\d=O_\e(\e^{3/2})$ then $\tilde{\v}_\e$ converges in $C^1_{\text{loc}}(\R^+)$ to the solution $\v_0$  of \eqref{eqlimitk>1/pi}.
\end{itemize}
\end{theorem}
 This description illustrates the switch of behaviour from one transition close to the outer boundary to many transitions, and thus many stripes.

\subsection{Ideas of the proofs}
In Section \ref{1}, we prove Theorems \ref{caskinf1/pi} and \ref{casksup1/pi}. In both cases we obtain the asymptotic expansion of the energy and then we study the behavior of minimizers. In the analysis of the minimizers of $F_{\b,\k}$, we strongly use the properties of the ODE satisfied by the minimizers, namely
\begin{equation}\label{eqint}
\v'(x)=\sqrt{\frac{1}{\b^2}\sin^2 \v(x) +\v'(0)^2}.
\end{equation}
There are two cases: one where $\v$ is in fact an almost solution of
\begin{equation}\label{eqintsimpl}\v'(x)=\sin \v(x)
\end{equation} which
is the usual Modica-Mortola solution, except that, in our case, it does not bridge $0$ to
$\pi$, but $0$ to $\arcsin 2\kappa \beta$ since we have (\ref{eqintsimpl})
and the Neumann condition at 1: $\v'(1)=2\k$. The usual Modica-Mortola
techniques allow to get an upper bound, lower bound, and expansion of
the energy. We prove that this case happens when $\kappa \b <1/\pi$ and in this case the ground state stays below $\pi$. On the other hand, when $\kappa\b>1/\pi$, we prove that a minimizer goes beyond $\pi$, and we even prove that it goes beyond $N\pi$ with $N$
large. The proof uses the equipartition of energy between the terms
$$\int_0^1 {\v'}^2 \hbox{ and } \int_0^1 \frac{1}{\b^2}\sin^2 \v(x)
+\v'(0)^2.$$ The definition of
 ${\tilde{\a}}_0$ comes from the minimization of the
 energy per period. This leads to the function $h$ defined by \eqref{defh2} in the proof of Proposition \ref{propcask>1surpi}.
To prove the convergence of minimizers of $F_{\b,\k}$ when $\b$ tends to zero we use appropriate bounds on the $H^1$ norm of $\v$ or of some of its blow-up versions, and this allows us to pass to the limit in the Euler-Lagrange equations satisfied by $\v_\b$. At the end of Section \ref{1} we also study the case when $\l$ goes to $+\infty$ and in a separate short subsection we give the asymptotic expansion of the simplified energy $F_{\b,\k}$ when $\l$ is close to $\frac{1}{\pi}$.

In Section \ref{2}, we prove Theorem \ref{minimizerspbcomplet} and some related further results. In order to do so, we first prove the uniform convergence of $v_\e$ to 1 as $\e$ goes to 0, which requires {$\d=o_\ep(\e)$}. We also prove that $\v_\e$ is an almost minimizer of $F_{\b,\k}$ with $\b=\e/\sqrt \d$ and that the full energy is given at leading order by $F_{\b,\k}(\v_\ep)$. We then study the behavior of minimizers by performing some blow-up and  passing to the limit in the Euler-Lagrange equations satisfied by the minimizers. In the case
$\tilde \k>1/\pi$, once we have proved that the limit of $\v_\e$ is quasi periodic, in order to find its  period, we need to prove that the energy per period of $\v_\e$ is almost minimizing. Thus we can deduce that its slope at the origin minimizes the function $h$ defined by \eqref{defh2} in the proof of Proposition
\ref{propcask>1surpi}.

\section{An auxiliary problem: minimization of the energy $F_{\b,\k}(\v)$}\label{1}

In this section, we study the simplified energy \eqref{eqF}.

\begin{proposition}\label{Propinitiale}
There exists a minimizer of $F_{\b,\k}$ in $\mathcal{J}$. Such a minimizer satisfies the following Euler-Lagrange equations:

\begin{equation}
\left\{
\begin{array}{lcll}\label{Equationsimplifie}
\v'' &=&\frac{1}{\b^2} \sin \v \cos \v \text{ in } (0,1), \\
\v(0)&=&0, \\
\v'(1)&=&2 \k.
\end{array}
\right.
\end{equation}
A minimizer $\v$ is smooth in $[0,1]$ and it also satisfies that for every $x$ in $[0,1]$
\begin{equation}\label{integrale1eree}
\v'(x)^2 =\frac{1}{\b^2}\sin^2 \v(x) +\v'(0)^2
\end{equation}
and $\v$ is increasing.
\end{proposition}

\begin{proof}
The existence and smoothness of minimizers are classical. Multiplying the first equation of  \eqref{Equationsimplifie} by $\v'$ and integrating we obtain \eqref{integrale1eree}. Now if $\v'(0)=0$ then $ \v \equiv 0$ in $[0,1]$ from the Cauchy-Lipschitz Theorem. This contradicts the fact that $\v'(1)=2\k$ with $\k>0$. We deduce that $\v'(0) \neq 0$ and thus $\v'$ does not vanish in $[0,1]$. Since $\v'(1)=2\k$ we obtain that $\v'>0$ in $[0,1]$.
\end{proof}

The rest of this section is devoted to the proof of Theorems \ref{caskinf1/pi} and \ref{casksup1/pi}. We recall the notation (\ref{tildekappa}).

\subsection{The case $\k<\frac{1}{\pi \b}$}

\begin{proposition}\label{DLenergie1}
For $\k< \frac{1}{\pi \b}$,  \eqref{DL11} holds.
If $\k$ is bounded as $\b$ goes to zero, or more generally if $\l=\k \b=o_\b(1)$, then we have \eqref{DL22}.
\end{proposition}

\begin{proof}
We first find a lower bound on the energy following the method of Modica-Mortola and then we construct a test function which gives the matching upper bound.

\textit{Lower bound:} By using Modica-Mortola's trick, we have that for $\v$ a minimizer of $F_{\b,\k}$ in $\mathcal{J}$:
\begin{eqnarray}
F_{\b,\k}(\v) & \geq &\frac{1}{4\b}\int_0^1|\v'||\sin \v|-\frac{\l}{2\b}\v(1) \nonumber
\end{eqnarray}
We use a change of variable and the fact that $\varphi$ is increasing to find that:
\[ F_{\b,\k}(\v)  \geq \frac{1}{4\b}\int_0^{\v(1)}|\sin y|\ dy - \frac{\l}{2\b}\v(1).\]
We let $N:=E( \frac{\v(1)}{\pi})$ where $E$ denotes the integer part. We obtain that
\[ F_{\b,\k}(\v) \geq \frac{N}{2\b}+\frac{1}{4\b}\int_{N\pi}^{\v(1)} |\sin y|\ dy-\frac{\l}{2\b}\v(1).\]
Since $y \mapsto |\sin y|$ is $\pi-$periodic and $\sin y \geq 0$ for $y$ in $[0,\pi]$, then
\[\int_{N\pi}^{\v(1)} |\sin y|\ dy= \int_0^{\v(1)-N\pi} \sin y \ dy= 1-\cos (\v(1)-N\pi).  \]
Thus
\begin{eqnarray}\label{firstlowerbound}
F_{\b,\k}(\v) \geq \frac{N}{2\b}(1-\l \pi)+\frac{1}{\b}\left[ \frac{1-\cos (\v(1)-N\pi)}{4}-\frac{\l}{2}(\v(1)-N\pi). \right]
\end{eqnarray}
Note that this first lower bound is valid for any $\k$. Now we study the function
\begin{equation}\label{deff}
f(x)=\frac{1-\cos x}{4}-\frac{\l x}{2}
\end{equation}
for $x$ in $[0,\pi]$. We have that $f$ is smooth, $f'(x)=\frac{\sin x}{4}-\frac{\l}{2}$ and $f''(x)=\frac{\cos x}{4}$. If $\l=\k \b\leq \frac {1}{2}$, then we set
\begin{equation}\label{defmin}
x_m^\b=\arcsin(2 \k \b) \in [0,\frac{\pi}{2}]
\end{equation} and $x_M^\b=\pi-\arcsin(2 \k \b)$.
We obtain that $f$ has a local minimum at $x_m^\b$ with
\begin{equation}\label{minn}
f(x_m^\b)= \frac{1-\sqrt{1-4\l^2}}{4}-\frac{\l}{2}\arcsin(2\l),
\end{equation}
besides $f$ has a maximum at $x_M^\b$ and $f(x_M^\b)=\frac{1+\sqrt{1-4\l^2}}{4}-\frac{\l}{2}(\pi-\arcsin(2\l))$. In order to know if the minimum of the function $f$ in $[0,\pi]$ is attained in $x_m^\b$ or in $\pi$ we set
\[ g(\l)=f(x_m^\b)-f(\pi)= \frac{1-\sqrt{1-4\l^2}}{4}-\frac{\l}{2}\arcsin(2\l)-\frac{1-\l \pi}{2}, \]
We observe that: $g(0)=-\frac{1}{2}$, $g(\frac{1}{2})=\frac{-1}{4}+\frac{\pi}{8}>0$ and $g'(\l)=\frac{\pi}{2}-\frac{\arcsin(2\l)}{2}>0$. Thus there exists a critical value of $\l$ called $\l_{\text{crit}} < \frac{1}{2}$ such that: if $\l < \l_{\text{crit}}$ then $f$ attains its unique minimum at $x_m^\b=\arcsin(2\l)$ and $\min_{[0,\pi]}f= \frac{1-\sqrt{1-4\l^2}}{4}-\frac{\l}{2}\arcsin(2\l)$, whereas if $ \l \geq \l_{\text{crit}}$ then $\min_{[0,\pi]}f= f(\pi)= \frac{1-\l \pi}{2}$. Note that $\l_{\text{crit}}> \frac{1}{\pi}$ since $g(\frac{1}{\pi})<0$.
Besides if $\l> \frac{1}{2}$, then $f'<0$ and $f$ is decreasing, meaning that $\min_{[0,\pi]}f =\frac{1-\l\pi}{2}$.\\
From this study, we obtain that for $\l < \frac{1}{\pi}$ we have
\begin{equation}\label{lowerbound}
F_{\b,\k}(\v) \geq \frac{N}{2\b}(1-\l \pi)+\frac{1}{\b}\left[\frac{1-\sqrt{1-4\l^2}}{4}-\frac{\l}{2}\arcsin(2\l)\right]
\end{equation}
for every {$\v$ minimizer of $F_{\b,\k}$} in $\J$ and $N=E(\frac{\varphi(1)}{\pi})$. \\

\textit{Upper bound:} We now construct a test function which shows that for $\l< \frac{1}{\pi}$ we have that a minimizer of $F_{\b,\k}$ in $\mathcal{J}$ satisfies $\v_\b(1)<\pi$ and the lower bound given by \eqref{lowerbound} with $N=0$ is optimal. For $\v$ in $\mathcal{J}$, we set $\psi(x)=\v(\b x)$ defined in $[0,\frac{1}{\b}]$ and we observe that
\[F_{\b,\k}(\v)=H_{\b,\k}(\psi):=\frac{1}{8\b}\int_0^{\frac{1}{\b}}\left(\psi'(x)^2+\sin^2 \psi(x) \right)dx -\frac{\l}{2\b}\psi(\frac{1}{\b}).\]
Let $\gamma>0$ be a small number to be fixed later ($\gamma \rightarrow 0$ as $\b \rightarrow 0$). In $(\gamma,\frac{1}{\b})$ we take $\psi(x)=2\arctan \left( e^{x-\frac{1}{\b}}\tan[\frac{\arcsin(2\l)}{2}]\right)$. This is the solution of
\begin{equation}
\left\{
\begin{array}{lcll}
\psi'&=&\sin \psi \ \text{ in } (\gamma,\frac{1}{\b}), \\
\psi'(\frac{1}{\b})&=&2\l, \\
\psi(\frac{1}{\b})&=& \arcsin(2\l).
\end{array}
\right.
\end{equation}
We thus have: $$ \int_\gamma^\frac{1}{\b} \psi'(x)^2+\sin^2\psi(x)\ dx=2\int_\gamma^\frac{1}{\b}|\psi'(x)||\sin \psi(x)|dx =2\int_{\psi(\gamma)}^{\psi(\frac{1}{\b})} |\sin y|\ dy,$$ where in the last equality we used the change of variable formula and the fact that $\psi$ is increasing. We set $\eta:=\psi(\gamma)= 2 \arctan (e^{\gamma-\frac{1}{\b}}\tan[\frac{1}{2}\arcsin (2 \l)])$ and in $[0,\gamma]$ we take $\psi(x)=\frac{\eta x}{\gamma}$.
We then have that
\begin{eqnarray}
H_{\b,\k}(\psi)&=&\frac{1}{\b}\left[ \int_0^\gamma \frac{\eta^2}{8\gamma^2}+\frac{1}{8}\sin^2(\frac{\eta x}{\gamma})dx-\frac{\l \eta}{2} \right] \nonumber \\
& & \phantom{aaaaaa } +\frac{1}{4\b} \int_\gamma^{\frac{1}{\b}} |\psi'(x)||\sin \psi(x)|dx -\frac{\l}{2\b}[\psi(\frac{1}{\b})-\eta]. \nonumber
\end{eqnarray}
It follows that
\begin{equation}\label{equationbeta}
H_{\b,\k}(\psi) \leq \left( \frac{\eta^2}{8\gamma}+\frac{\gamma}{8}\right)\frac{1}{\b}+ \frac{\cos \eta - \sqrt{1-4\l^2}}{4\b}-\frac{\l}{2\b}\arcsin(2 \l).
\end{equation}
We then choose $\gamma$ such that $\frac{\gamma}{\beta} \rightarrow 0$ and $\frac{\eta^2}{\gamma \beta}\rightarrow 0$ (we can take $\gamma=\beta^n$ for all $n \geq 2$) and we obtain
\begin{equation}\label{upperbound}
\min_{\mathcal{J}} F_{\b,\k} \leq H_{\beta,\k}(\psi) \leq \frac{1}{\b}\left[\frac{1-\sqrt{1-4\l^2}}{4}-\frac{\l}{2}\arcsin(2\l)\right] +o_\beta(\beta^n), \text{ for all } n \text{ in } \mathbb{N}^*.
\end{equation}
By using \eqref{lowerbound} and \eqref{upperbound} together we find that for $\l < \frac{1}{\pi}$, a minimizer $\v_\b$ of $F_{\b,\k}$ satisfies  $N:=E(\frac{\v_\b(1)}{\pi})=0$ and

\begin{equation}
\min _{\v \in \J} F_{\b,\k}(\v)= \frac{1-\sqrt{1-4\l^2}}{4\b}-\frac{\l}{2\b}\arcsin(2\l) +o_\b(\beta^n) \text{ for all } n \text{ in } \mathbb{N}^*.
\end{equation}
We recall that $\k=\frac{\l}{\b}$ and this yields that, if $\k \b =o_\b(1)$ we have
\begin{eqnarray}
\min_{\mathcal{J}} F_{\b,\k} &=&\frac{1-\sqrt{1-4\k^2 \b^2}}{4\b}-\frac{\k}{2}\arcsin(2\k \b)+o_\b(\b^2) \nonumber \\
&=& \left( 1-(1-2\k^2\b^2)\right)\frac{1}{4\b}-\k^2\b+o_\b(\k^2\b). \nonumber
\end{eqnarray}
Hence if $\k$ is fixed or if $\k=o_\b(\frac{1}{\b})$ then we obtain \eqref{DL22}.
\end{proof}

\begin{proposition}\label{valuephi1}
For $\l=\k \b <\frac{1}{\pi}$, with $\l$ independent of $\b$, let $\v_\b$ be a minimizer of $F_{\b,\k}$ in $\mathcal{J}$. We have that
\begin{equation}\label{eq:valuephi1}
\v_\b(1)=\arcsin(2 \k \b) +o_\b(\b^n)
\end{equation} for all $n$ in $\mathbb{N}^*$. In particular for $\beta$ small enough we have $0\leq \v_\b(x) \leq \v_\b(1)< \frac{\pi}{2}$. Besides if $\l=o_\b(1)$, we have that $\v_\b(1)=2\k \b(1+o_\b(1))$.
\end{proposition}
\begin{proof}
From \eqref{firstlowerbound}, where we know that $N=0$, and \eqref{upperbound}, we deduce that for $\v_\b$ a minimizer of $F_{\b,\k}$ in $\mathcal{J}$ we have
\begin{equation}\label{num1}
\frac{1-\cos \v_\b(1)}{4\b}-\frac{\l\v_\b(1)}{2\b}\leq F_{\b,\k}(\v_\b) \leq \frac{1}{\b}\left[\frac{1-\sqrt{1-4\l^2}}{4}-\frac{\l}{2}\arcsin(2\l)\right]+o_\b(\b^n)
\end{equation}
for all $n$ in $\mathbb{N}^*$.
With $f$ defined as \eqref{deff},  we have $\frac{1-\sqrt{1-4\l^2}}{4}-\frac{\l}{2}\arcsin(2\l)=\min_{[0,\pi]} f=f(x_m^\b)$, where $x_m^\b=\arcsin(2 \k\b)$. Then (\ref{num1}) implies
\begin{equation}\label{fmin}f(\v_\b(1))\leq f(x_m)+o_\b(\b^n).
\end{equation}
\begin{itemize}
\item[1)] If $\l=\k \b$ is independent of $\b$, then $f$ and $x_m:=x_m^\b=\arcsin(2 \l)$ do not depend on $\b$. From \eqref{fmin}, we obtain that
\[ f(x_m)\leq f(\v_\b(1))\leq f(x_m)+o_\b(\b^n)\]
for all $n$ in $\mathbb{N}^*$. The study of the function $f$ done in the proof of the previous proposition then shows that $\v_\b(1) \rightarrow \arcsin(2 \l)$.  Expanding $f$ around $x_m$, we have that
\begin{equation}\label{expansion}
f(\v_\b(1))=f(x_m)+\frac{f''(x_m)}{2}(\v_\b(1)-x_m)^2+o_\b[(\v_\b(1)-x_m)^2].
\end{equation}
This proves that  $(\v_\b(1)-x_m)=o_\b(\b^n)$ for all $n$ in $\mathbb{N}$.
\item[2)] If $\l=\k \b=o_\b(1)$, from \eqref{num1} we still have that
\[ f(\v_\b(1))=f(x_m^\b)+o_\b(\b^n).\]
besides we observe that $f$ defined by \eqref{deff} converges uniformly to $f_0(x)=\frac{1-\cos x}{4}$ on $[0,\pi]$. Since $f(x_m^\b)$ goes to zero as $\b$ goes to zero this implies that $\v_\b(1)\rightarrow 0$ as $\b\rightarrow 0$. Now we can write
 \begin{equation}\nonumber
f_\b(\v_\b(1))=f(x_m^\b)+\frac{f''(x_m^\b)}{2}(\v_\b(1)-x_m^\b)^2+o_\b[(\v_\b(1)-x_m^\b)^2].
\end{equation}
Since $f''(x_m^\b) \rightarrow \frac{1}{4}$ we conclude that $\v_\b(1)=x_m^\b+ o_\b(\b^n)$ for all $n$ in $\mathbb{N}^*$. Expanding $x_m^\b=\arcsin(2\k \b)=2\k \b +o_\b(\k\b)$, we conclude the proof.
\end{itemize}
\end{proof}

\begin{proposition}
Let $\l=\k \b< \frac{1}{\pi}$, if $\l$ is independent of $\b$ or if $\l=o_\b(1)$, then for $\b$ small enough, there exists a unique minimizer of $F_{\b,\k}$ in $\mathcal{J}$.
\end{proposition}

\begin{proof}
Let $\v_\b$ be a minimizer of $F_{\b,\k}$ in $\mathcal{J}$. From Proposition \ref{Propinitiale} and Proposition \ref{valuephi1}, we know that $\v_\b$ is increasing and that $0\leq \v_\b(x) \leq \v_\b(1)<\frac{\pi}{2}$, for $\b$ small enough. For simplicity, we let $\v=\v_\b$ and we let $\a:=\v'(0)$.
We observe that \eqref{integrale1eree} implies that $4\l^2-\b^2\a^2\geq 0$ and since we know that $\v(1)< \frac{\pi}{2}$ we deduce that $\v(1)=\arcsin [4\l^2-\b^2\a^2]$. Taking the square root of \eqref{integrale1eree} we obtain
$ \v'(x)=\frac{\sqrt{\sin^2 \v(x) + \b^2\a^2}}{\b}$ for all $x$ in $[0,1]$. This implies
\begin{equation}
g(\a):= \beta \int_0^{ \arcsin [4\l^2-\b^2\a^2]} \frac{dy}{\sqrt{\sin^2 y +\b^2\a^2}}=1.
\end{equation}
We claim that there exists a unique $\a>0$ such that $g(\a)=1$. This will imply uniqueness of the minimizer $\v$ by the Cauchy-Lipschitz Theorem. To prove our claim we observe that $g$ is smooth, $g(0)=+\infty$, $g(2\k)=0$ and

\begin{eqnarray}
g'(\a) &=&\beta \int_0^{\arcsin [4\l^2-\b^2\a^2]} \frac{-\b^2 \a\ dy}{(\sin^2y+\b^2 \a^2)^{3/2}} \nonumber \\
  & & \phantom{aaaaaaaa}- \frac{2\b\a}{\sqrt{1-4\l^2+\b^2\a^2}}\times \frac{1}{(4\l^2-\b^2\a^2)^2+\b^2\a^2}<0.
\end{eqnarray}
This concludes the proof.
\end{proof}

The next Proposition states that the minimizer of $F_{\b,\k}$ converges exponentially fast to zero away from the point $1$ as $\b$ converges to zero.

\begin{proposition}\label{decroissanceexp}
Let $\k \b< \frac{1}{\pi}$, with $\l$ independent of $\b$ or $\l=o_\b(1)$ and let  $\v_\b$ be the minimizer of $F_{\b,\k}$  in $\mathcal{J}$. Then \eqref {eqdecroissanceexp} holds and
 $\v_\b \rightarrow 0$ in $\mathcal{C}^\infty_{\text{loc}}([0,1))$.
\end{proposition}

\begin{proof}
 It follows from \eqref{integrale1eree} that for every $x$ in $[0,1)$ we have:
\begin{eqnarray}
\v'(x)^2 > \frac{\sin^2 \v(x)}{\b^2}  \text{ and } \v'(x)> \frac{| \sin \v(x)|}{\b}. \nonumber
\end{eqnarray}
Since $0<\v(x) < \frac{\pi}{2}$, for $0<x\leq 1$ and for $\b$ small enough, from Proposition \ref{valuephi1}, we can say that $\frac{\v'(x)}{\sin \v(x)} > \frac{1}{\b}$ for every $x$ in $(0,1)$. Integrating this relation between $x$ and $1$ yields
\begin{eqnarray}
 \log \frac{\tan (\v(1)/2)}{\tan(\v(x)/2)} > \frac{1-x}{\b}  \nonumber \\
 \Rightarrow  \tan \frac{\v(x)}{2} < \tan \frac{\v(1)}{2} e^{\frac{x-1}{\b}}
\end{eqnarray}
for every $x$ in $[0,1)$ and this implies \eqref{eqdecroissanceexp}. To deduce that $\v$ converges to zero in $\mathcal{C}^\infty_{\text{loc}}([0,1))$ we observe that from the first Equation of \eqref{Equationsimplifie}, we have that $\v''$ tends to zero in $\mathcal{C}^0(K)$ for every compact set $K \subset[0,1)$. Let us show that $\varphi'(0)$ converges to zero.  From \eqref{num1} and \eqref{eq:valuephi1} we find that for all $n$ in $\mathbb{N}$:
\begin{eqnarray}
 \frac{1-\sqrt{1-4\l^2}}{4\b}+o_\b(\b^n)&=& \frac{1-\cos \varphi(1)}{4\b}= \frac{1}{4\b}\int_0^1 |\varphi'||\sin\varphi| \nonumber \\
 &\leq & \frac{1}{8}\int_0^1\varphi'^2+\frac{1}{8\b^2}\int_0^1\sin^2\varphi \nonumber \\
 & \leq & \frac{1-\sqrt{1-4\l^2}}{4\b}+o_\b(\b^n). \nonumber
\end{eqnarray}
In particular we find that for all $n$ in $\mathbb{N}$ we have
\begin{equation}\label{varphi'}
\int_0^1\left(\varphi'-\frac{\sin \varphi}{\b} \right)^2=o_\b(\b^n).
\end{equation}
Now, by using \eqref{integrale1eree}, we find that
\begin{equation}\label{varphi'2}
\left(\varphi'(x)-\frac{\sin \varphi(x)}{\b} \right)^2 \geq \frac{\b^2 \varphi'(0)^4}{\sqrt{1+\b^2\varphi'(0)^2}+1}\geq \b^2\varphi'(0)^4(\frac{1}{2}+o_\b(1)).
\end{equation}
By using \eqref{varphi'} and \eqref{varphi'2} we find $\varphi'(0) \rightarrow 0 $ as $\b$ tends to zero. This  implies that $\varphi'$ converges to zero in $\mathcal{C}^0(K)$ for every compact set $K \subset[0,1)$. Then a classical bootstrap argument allows us to infer that  $\v$ converges to zero in $\mathcal{C}^\infty_{\text{loc}}([0,1))$.
\end{proof}

Now that we know the behaviour of the minimizer on every compact set of $[0,1)$ we study the shape of the minimizer near the point 1. We begin with the case $\k \b=o_\b(1)$.

\begin{proposition}
Let us assume that $\k \b =o_\b(1)$, then for $\b$ small enough, the minimizer $\v_\b$ of $F_{\b,\k}$ satisfies $\|\v_\b\|_{L^\infty([0,1])} \leq C\k \b,$  for some  $C>0$  independent of $\b$.
We set $\Phi_\b(x):=\frac{\v_\b(1-\b x)}{2\k \b}$, defined in $[0,\frac{1}{\b})$. We have that $\Phi_\b \rightarrow \Phi_0:= e^{-x}$ in $\mathcal{C}^\infty_{\text{loc}}([0,+\infty)]$.
\end{proposition}

Before proceeding to the proof of this proposition, we remark that, in general, although $\v_\b$ converges to zero in $\mathcal{C}^0([0,1])$, we do not have $\mathcal{C}^1$ convergence of $\v_\b$ on $[0,1]$ since $\v_\b'(1)=2\k$.

\begin{proof}
Since $\v_\b$ is increasing in $[0,1]$ and since $\v_\b(1)=2\k \b (1+o_\b(1))$ from Proposition \ref{valuephi1}, we deduce that $\|\v_\b\|_{L^\infty([0,1])} \leq C\k \b$ for $\b$ small enough, with $C$ independent of $\b$ and $\k$.
Now we observe that $\Phi_\b=\frac{\v_\b(1-\b x)}{2\k \b}$ satisfies:
\begin{equation}\label{equationsurpsi}
\left\{
\begin{array}{lcll}
\Phi_\b''&=&\cos(2\k \b \Phi_\b) \frac{\sin(2\k \b \Phi_\b)}{2\k\b} \text{ in } (0,\frac{1}{\b}), \\
\Phi_\b(0)&=&\frac{\v_\b(1)}{2\k \b}, \\
\Phi_\b'(0)&=&1.
\end{array}
\right.
\end{equation}
Besides,
\[ \int_0^{\frac{1}{\b}} \Phi_\b '(x)^2dx= \frac{1}{2\k^2}\int_0^{\frac{1}{\b}} \v_\b '(1-\b x)^2dx=\frac{1}{2\k^2 \b}\int_0^1 \v_\b '(y)^2\ dy.\]
We recall that $F_{\b,\k}(\v_\b)=\frac{-\k^2 \b}{2}(1+o_\b(1))$ and $\v_\b(1)=2\k \b(1+o_\b(1))$ if $\k \b=o_\b(1)$. From this, we deduce that $\int_0^1 \v_\b'(y)^2\ dy \leq C \k^2 \b$ and $\int_0^{\frac{1}{\b}} \Phi'_\b(x)^2dx \leq C$. Thus, since $\Phi_\b(0) \rightarrow 1$, $\Phi_\b$ is bounded in $H^1_{\text{loc}}(\R^+)$. From the Sobolev injections, up to a subsequence in
$\b$, we have $\Phi_\b \rightarrow \Phi_0$ in $\mathcal{C}^0_{\text{loc}}([0,+\infty))$ for some $\Phi_0$ in  $\mathcal{C}^0_{\text{loc}}([0,+\infty))$. We also have that $2\k \b \Phi_\b \rightarrow 0$ in $\mathcal{C}^0_{\text{loc}}([0,+\infty)$ and $\frac{\sin(2\k \b \Phi_\b)}{2\k \b \Phi_\b}\rightarrow 1$ in $\mathcal{C}^0_{\text{loc}}([0,+\infty)$. Therefore, we can pass to the limit in \eqref{equationsurpsi} and find that $\Phi_0$ satisfies
\[\Phi_0''=\Phi_0 \ \text{ in } \R^+.\]
Since $\Phi''_\b$ is bounded in $L^\infty_{\text{loc}}([0,+\infty))$ we have that
$\Phi_\b$ converges to $\Phi_0$ in $\mathcal{C}^1_{\text{loc}}([0,+\infty))$. In particular,  $\Phi_0(0)=1$ and $\Phi_0'(0)=1$, that is $\Phi_0(x)=e^{-x}$. By uniqueness of the limit, the entire sequence converges and by using a bootstrap argument we can show that the convergence holds in $\mathcal{C}^\infty_{\text{loc}}([0,+\infty))$.
\end{proof}

We now study the case where $\l=\k \b$ is independent of $\b$ and $\l>\frac{1}{\pi}$. Note that in this case $\v_\b$ does not converge to zero in $\mathcal{C}^0([0,1])$ since $\v_\b(1)=2\arcsin(2\l)+o_\b(1)$ from Proposition \ref{valuephi1}. The transition from 0 to $2\arcsin(2\l)$ takes place in a boundary layer of size $\beta$:

\begin{proposition}\label{blowblowup}
Let $\v_\b$ be the minimizer of $F_{\b,\k}$ in $\mathcal{J}$. We set $\psi_\b(x):=\v_\b(1-\b x)$ defined in $[0,\frac{1}{\b}]$. We then have that $\psi_\b \rightarrow \psi_0$ in $\mathcal{C}^\infty_{\text{loc}}([0,+\infty))$ where $\psi_0(x)=2\arctan \left[ \tan \left(\frac{\arcsin(2\l)}{2} \right)e^{-x}\right]$.
\end{proposition}
\begin{proof}
The function $\psi_\b=\v_\b(1-\b x)$ satisfies
\begin{equation}\label{blowup1}
    \left\{
    \begin{array}{lcll}
    \psi_\b''&=&\sin \psi_\b \cos \psi_\b \text{ in } (0,\frac{1}{\b}), \\
    \psi_\b'(0)&=&-2\l, \\
    \psi_\b(\frac{1}{\b})&=&0.
    \end{array}
    \right.
    \end{equation}
We have that $F_{\b,\k}(\v_\b)= \frac{1}{8\b}\int_0^{\frac{1}{\b}}\left(\psi_\b'(x)^2+\sin^2\psi_\b(x)\right) dx-\frac{\l}{2\b}\psi_\b(0)$. From Proposition \ref{DLenergie1} and \ref{valuephi1}, we deduce that
\begin{equation}
\frac{1}{8\b}\int_0^{\frac{1}{\b}}\left(\psi_\b'(x)^2+\sin^2\psi_\b(x)\right) dx-\frac{\l}{2\b}\psi_\b(0)=\frac{1-\sqrt{1-4\l^2}}{4\b}-\frac{\l}{2\b}\arcsin(2\l)+o_\b(1),
\end{equation}
and $\psi_\b(0)=\v_\b(1)=\arcsin(2\l)+o_\b(1)$.
Hence
\begin{equation}\label{boundeontheblowup}
\int_0^{\frac{1}{\b}}\left(\psi_\b'(x)^2+\sin^2\psi_\b(x)\right) dx \leq 2( 1-\sqrt{1-4\l^2})+o_\b(\b).
\end{equation}
This proves that $\psi_\b$ is bounded in $H^1_\text{loc}(\R^+)$ and hence converges weakly  to some $\psi_0 \in H^1_{\text{loc}}(\R^+)$, up to a subsequence in $\b$. Note that this also implies that the convergence is uniform on every compact set of $\R^+$. Now using \eqref{boundeontheblowup} and some lower semi-continuity result, we obtain
\begin{equation}\label{boundonthelimit}
\int_0^{M} \psi_0'(x)^2+\sin^2\psi_0(x) dx \leq C_{\l}
\end{equation}
for every $M>0$  and where $C_{\l}$ is a constant which depends on $\l$. Thus we have
$\int_0^{+\infty} \psi_0'(x)^2+\sin^2\psi_0(x) dx \leq C_{\l}$. The uniform convergence on every compact set allows us to pass to the limit in the sense of distributions in the first equation of \eqref{blowup1}, that is: $\psi_0$ satisfies $\psi_0''=\sin \psi_0\cos \psi_0$ in $\R^+$. Thus from the regularity theory $\psi_0$ is $\mathcal{C}^\infty(\R^+)$. From \eqref{boundonthelimit} we deduce that
\begin{equation}\label{psi0equal0}
\lim_{x\rightarrow +\infty}\psi_0'(x)=0 \ \ \ \text{ and } \lim_{x\rightarrow +\infty} \sin^2 \psi_0(x)= 0.
\end{equation}
The function $\psi_0$ also satisfies $\psi_0'(x)^2=\sin^2 \psi_0(x)+C$ for $x$ in $\R^+$, with $C$ a constant. Equation \eqref{psi0equal0} proves that $C=0$. Thanks to the first equation of \eqref{blowup1} and a bootstrap argument we also have that the convergence is smooth on every compact set of $[0,+\infty)$. In particular $\psi_0'(0)=-2 \l$ and we have that $\psi_0$ is decreasing. Now recall from Proposition \ref{valuephi1} that $0\leq \psi_\b(x) <\frac{\pi}{2}$ for $x$ in $\R^+$. This implies that $\psi_0 '=-\sin \psi_0$, $\psi_0'(0)=-2 \l$ and $\psi_0(0)=\arcsin(2\l)$  and thus $\psi_0(x)=2\arctan \left[ \tan \left(\frac{\arcsin(2\l)}{2} \right)e^{-x}\right]$.
\end{proof}
Theorem \ref{caskinf1/pi} follows from the Propositions of this section.

\begin{proposition}
Let $\k \b=\frac{1}{\pi}$ then \eqref{DL11} still holds.
\end{proposition}

\begin{proof}
Indeed coming back to the proof of Proposition \ref{DLenergie1} we find that the lower bound \eqref{lowerbound} and the upper bound \eqref{upperbound} remain true in that case, so does the expansion of the ground state of the energy.
\end{proof}
We are not able to give the behaviour of minimizers of $F_{\b,\k}$ when $\b$ goes to $0$ in the case $\l=1/\pi$ although we suspect that in this case we have a unique minimizer and it has the same behaviour as minimizers for $\l<\frac{1}{\pi}$ (cf.\ Propositions \ref{decroissanceexp} and \ref{blowblowup}). The main difficulty is that if $\l=\frac{1}{\pi}$, then the lower bound \eqref{lowerbound} and the upper bound \eqref{upperbound} do not imply that $N=0$. Thus we could have $\v_\b(1)=N_\b+\arcsin(\frac{2}{\pi})$ with $N_\b$ an integer which can be unbounded as $\b$ goes to $0$.

\subsection{The case $ \k> \frac{1}{\pi \b}$}

In this section, we will see that a change of regime occurs when $\l =\k \b>\frac{1}{\pi }$, in the sense that the  minimizer of $F_{\b,\k}$ makes several transitions from $0$ to $\pi$, from $\pi$ to $2\pi$ etc.  The first step is to prove that a minimizer $\v_\b$ satisfies $\v_\b(1)\geq \pi$. This is true as soon as $\l> \frac{1}{\pi}$ even if $\l$ depends on $\b$.

\begin{lemma}\label{Superieurapi}
Let $\l=\k \b > \frac{1}{\pi}$,  and let $\v_\b$ be a minimizer of $F_{\b,\k}$ in $\mathcal{J}$, then $\v_\b(1) \geq \pi.$
\end{lemma}

\begin{proof}
Let us call $N=E(\frac{\v_\b(1)}{\pi})$. We recall that $f$ given by \eqref{deff} satisfies from \eqref{firstlowerbound} that:
if $\frac{1}{\pi}<\l <\l_{\text{crit}}$ then $f(x)\geq \frac{1-\sqrt{1-4\l^2}}{4}-\frac{\l}{2}\arcsin(2 \l)$ whereas if $\l \geq \k_{\text{crit}}$, then $f(x)\geq \frac{1-\l \pi}{2}$.
\begin{itemize}
\item[1)] If $\l \geq \l_{\text{crit}}$,  then
\[F_{\b,\k}(\v_\b)\geq \frac{N+1}{2\b}(1-\l \pi). \]
We claim that we can construct a sequence ${(\psi_\b)}_\b$ such that $\psi_\b$ is in $\mathcal{J}$ and $\limsup_{\b \rightarrow 0} \beta F_{\b,\k}(\psi_\b) \leq (1-\l \pi)$. This will imply that $N\geq 1$.
For the construction of such a sequence we refer to \cite{Modica1987} or p.106-107 of \cite{Braides2002} and we just sketch the argument here. Let $\v_0$ be the solution of $\v_0'=\sin \v_0$ such that $\v_0(-\infty)=0$ and $\v_0(+\infty)=\pi$. This solution is the minimizer of
\[ \min\{ \int_{-\infty}^{+\infty} \left(\v'(x)^2+\sin^2 \v(x)\right)dx; \ \v(-\infty)=0 \text{ and } \v(+\infty)=\pi \}.\]
Besides it satisfies $\int_{-\infty}^{+\infty} \left(\v_0'(x)^2+\sin^2 \v_0(x)\right)dx=4$.
Now we let $T>0$ and $\v_T$ be a minimizer of
\[ \min\{ \int_{-T}^{+T} \left(\v'(x)^2+\sin^2 \v(x)\right)dx; \ \v(-T)=0 \text{ and } \v(+T)=\pi \}.\]
We have that $\v_T \rightarrow \v_0$ in $H^1_{\text{loc}}(\R)$  and
\[  \int_{-T}^{+T} \left(\v_T'(x)^2+\sin^2 \v_T(x)\right)dx \rightarrow \int_{-\infty}^{+\infty} \left(\v_0'(x)^2+\sin^2 \v_0(x)\right)dx\]
as $T \rightarrow +\infty$.
To construct our sequence we choose $0<t_1<t_2<1$ and we set
\begin{equation}
\psi_\b(x):= \begin{cases}
0 & \text{ if }  0\leq x < t_1-\b T,  \\
\v_T(\frac{x-t_1}{\b}) & \text{ if }  t_1-\b T \leq x \leq t_1+\b T, \\
\pi & \text{ if }  t_1+\b T < x <t_2-\b T, \\
\v_T(\frac{x-t_2}{\b})+\pi & \text{ if }  t_2-\b T \leq x \leq t_2+\b T, \\
2\pi & \text{ if } t_2+\b T <x \leq 1.
\end{cases}
\end{equation}
We then have that

\begin{eqnarray}
\b F_{\b,\k}(\psi_\b)= \frac{1}{8}\int_{\{|x-t_1|\leq \b T \}}\b \v_T'^2(\frac{x-t_1}{\b})+\frac{1}{\b} \sin^2 \v_T(\frac{x-t_1}{\b})dx \nonumber \\
\phantom{aaaaa}+\frac{1}{8}\int_{\{|x-t_2|\leq \b T \}} \b \v_T'^2(\frac{x-t_2}{\b})+\frac{1}{\b} \sin^2 \v_T(\frac{x-t_2}{\b})dx -\l \pi. \nonumber
\end{eqnarray}
We then make a change of variable $y=\frac{x-t_1}{\b}$ (or $y=\frac{x-t_2}{\b}$) and we obtain
\begin{eqnarray}
\b F_{\b,\k}(\psi_\b)= \frac{1}{8}\int_{\{|y|\leq  T \}} \left(\v_T'^2(y)+ \sin^2 \v_T(y)\right)\ dy \nonumber \\
\phantom{aaaa}+\frac{1}{8}\int_{\{|y|\leq T\}} \left(\v_T'^2(y)+ \sin^2 \v_T(y)\right)\ dy -\l \pi. \nonumber
\end{eqnarray}
We take $T\rightarrow +\infty$ (but keeping in mind that $\b T$ must satisfy $1-t_2<\b T$) and that proves that $\limsup_{\b \rightarrow 0} \beta F_{\b,\k}(\psi_\b) \leq (1-\l \pi)$.

\item[2)] If $\frac{1}{\pi}< \l < \l_{\text{crit}}$ then \eqref{firstlowerbound} implies
\[ F_{\b,\k}(\v_\b) \geq \frac{N}{2\b}(1-\k \pi) + \frac{1-\sqrt{1-4\l^2}}{4\b}-\frac{\l}{2}\arcsin(2 \l).\]
We claim that we can construct a sequence $(\psi_\b)_\b$ such that $F_{\b,\k}(\psi_\b)=\
\frac{1}{2\b}(1-\k \pi) + \frac{1-\sqrt{1-4\l^2}}{4\b}-\frac{\l}{2}\arcsin(2 \l)+o_\b(1)$. This sequence is built combining the previous construction with the construction of the test function of the proof of Proposition \ref{DLenergie1}. More precisely we take:
\begin{equation}
\psi_\b(x):= \begin{cases}
0 & \text{ if }  0\leq x < \frac{1}{4}-\b T,  \\
\v_T(\frac{x-t_1}{\b}) & \text{ if }  t_1-\b T \leq x \leq t_1+\b T, \\
\pi & \text{ if }  t_1+\b T < x <\frac{1}{2}, \\
\pi+ \frac{(x-1/2)}{\b}\eta & \text{ if } \frac{1}{2}\leq x \leq \frac{1}{2}+\b, \\
\pi + \arctan\left(e^{\frac{x-1}{\b}}\tan[\arcsin(2\k \b)]\right) & \text{ if } \frac{1}{2}+\b\leq x \leq 1,
\end{cases}
\end{equation}
with $\eta:= 2\arctan\left( e^{\frac{1}{2\b}}\tan[\arcsin(2\k \b)]\right)$. Note that we have $\psi_\b'=\frac{\sin \psi_\b}{\b}$ in $( \frac{1}{2}+\b,1)$ and that $\psi_\b(1)=\pi+\arcsin(2\l)$. By combining the previous point with the ideas of the construction of the test function in the proof of Proposition \ref{DLenergie1} we can conclude the proof.
\end{itemize}
\end{proof}

The next proposition shows that a minimizer of $F_{\b,\k}$ in $\mathcal{J}$ enjoys some symmetry property with respect to the point $\frac{T}{2}$ such that $\v_\b(\frac{T}{2})=\frac{\pi}{2}$ and also some periodicity property of period $T$.

\begin{proposition}\label{Symmetry}
Let $\l > \frac{1}{\pi}$ and let $\v_\b$ be a minimizer of $F_{\b,\k}$ in $\mathcal{J}$. Let $\frac{T}{2}$ be the point in $(0,1)$ such that $\v_\b(\frac{T}{2})=\frac{\pi}{2}$. Then we have that
\begin{equation}\label{symmetry}
\v_\b(x)=\pi-\v_\b(T-x), \ \text{ for all } x \text{ in } [0,T],
\end{equation}
and in particular $\v_\b(T)=\pi$. Furthermore,  $\v_\b$ is periodic in the following sense:
\begin{equation}\label{periodicity}
\v_\b(x+T)=\pi+\v_\b(x), \text{ for all } x \text{ in } [0,1-T]
\end{equation}
\end{proposition}

\begin{proof}
First note that from Lemma \ref{Superieurapi} and from the fact that $\v_\b$ is increasing, we have the existence and uniqueness of $\frac{T}{2}$ such that $\v_\b(\frac{T}{2})=\frac{\pi}{2}$.
We set $\psi_\beta(x) :=\pi-\v_\b(T-x)$ defined for $0\leq x\leq T$.
Then the Cauchy-Lipschitz Theorem implies that $\psi_\b=\v_\b$. Taking $x=T$ in \eqref{symmetry} we find $\v_\b(T)=\pi$.
In the same way, we now set $\Phi_\b(x):=\v_\b(x+T)-\pi$ defined for $x$ in $[0,1-T]$.
We apply the Cauchy-Lipschitz Theorem again and find \eqref{periodicity}.
\end{proof}
We can now obtain an expansion of the ground state of the energy when $\l$ does not depend on $\b$.

\begin{proposition}\label{propcask>1surpi}
Let $\l=\k \b>\frac{1}{\pi}$ with $\l$ independent of $\b$. Let $\v_\b$ be a minimizer of $F_{\b,\k}$ in $\mathcal{J}$, let $0<{\tilde{\a}}_0<2\l$ be the unique number such that \eqref{defalpha0} holds, then we have the asymptotic expansion \eqref{DLenergie2}.
\end{proposition}
\begin{proof}
To prove the expansion of the energy we first find a suitable lower bound for the energy.\\

\textit{Lower bound:} Note that we have $\v_\b(0)=0$, $\v_\b(\frac{T}{2})=\frac{\pi}{2}$. We let $\a:= \v_\b'(0)$. For simplicity we let $\v=\v_\b$, from \eqref{integrale1eree} integrated from $0$ to $T/2$, with the help of a change of variable we find
\begin{eqnarray}\label{defT}
\frac{T}{2} =\int_0^{\frac{\pi}{2}} \frac{\b\  dy}{\sqrt{\a^2 \b^2+\sin^2 y}}.
\end{eqnarray}
We define $N$ by the relation $N=E(\frac{1}{T})$. From the quasi-periodicity property \eqref{periodicity} and the fact that $\v$ is increasing we have that:
\begin{equation}\nonumber
F_{\b,\k}(\v) \geq 2N e- \frac{\l (\v(1)-N\pi)}{2\b}
\end{equation}
where $e$ is the minimum of the energy on half of a period, that is
\begin{equation}\label{defe}
e := \min_{ u \in \mathcal{F}} \int_0^{\frac{T}{2}} \left(\frac{u'(x)^2}{8}+\frac{1}{8\b^2}\sin^2u(x)\right)dx-\frac{\l \pi }{4\b}
\end{equation}
and
\begin{equation}\label{defF}
\mathcal{F}:= \{ u \in H^1((0,\frac{T}{2})); \ u(0)=0 \text{ and } u(\frac{T}{2})=\frac{\pi}{2} \}.
\end{equation}
We now write
\[ \int_0^{\frac{T}{2}} \left(\frac{u'(x)^2}{8}+\frac{\sin^2u(x)}{8\b^2}\right)dx-\frac{\l \pi }{4\b}= \int_0^{\frac{T}{2}} \left[\frac{u'(x)^2}{8}+\left(\frac{\sin^2u(x)}{8\b^2} +\frac{\a^2 }{8}\right)\right]dx-\frac{\l \pi }{4\b}-\frac{\a^2 T}{16} \]
and we use Modica-Mortola's trick to say that
\begin{equation} e \geq \frac{1}{\b} \left[ \int_0^{\frac{\pi}{2}} \frac{1}{4}\sqrt{\a^2 \b^2 +\sin^2y }\ dy -\frac{\l \pi}{4} \right]-\frac{\a^2  T}{16 }.
\end{equation}
We now set $\tilde{\a}= \a \b$. We claim that $e<0$ (we postpone the proof of this fact for clarity and refer to Lemma \ref{enegative}). Since $N=E(\frac{1}{T}) \leq \frac{1}{T}$ we thus have that
\begin{equation}\label{bbb}
F_{\b,\k}(\v) \geq \frac{2 e}{T}- \frac{\l (\v(1)-N\pi)}{2\b}.
\end{equation}
We use \eqref{defT} to estimate  $\frac{2e\b^2}{T}$ and find:
\begin{equation}\label{defh} \frac{2e \b^2}{T} \geq \frac{ \int_0^{\frac{\pi}{2}}\sqrt{\tilde{\a}^2  +\sin^2y }\ dy- \l\pi  }{4\int_0^{\frac{\pi}{2}} \frac{ dy}{\sqrt{\tilde{\a}^2 +\sin^2 y}} } -\frac{\tilde{\a}^2}{8}.
\end{equation}
Let us study the function
\begin{equation}\label{defh2}
h(x):= \frac{ \int_0^{\frac{\pi}{2}}\sqrt{x^2  +\sin^2y }\ dy- \l\pi  }{4\int_0^{\frac{\pi}{2}} \frac{ dy}{\sqrt{x^2 +\sin^2 y}} } -\frac{x^2}{8}.
\end{equation}
This function is smooth for $x>0$ and we have
\begin{eqnarray}\nonumber
h'(x)&=& \frac{ x\left(\int_0^{\frac{\pi}{2}} \frac{dy}{\sqrt{x^2+\sin^2y}}\right)^2+x \left( \int_0^{\frac{\pi}{2}} \sqrt{x^2+\sin^2y}\ dy -\l \pi  \right) \int_0^{\frac{\pi}{2}} \frac{dy}{(\sqrt{x^2+\sin^2y})^3}}{4\left(\int_0^{\frac{\pi}{2}} \frac{dy}{\sqrt{x^2+\sin^2y}}\right)^2}-\frac{x}{4} \\
&=& \frac{x \left( \int_0^{\frac{\pi}{2}} \sqrt{x^2+\sin^2y}\ dy -\l \pi  \right) \int_0^{\frac{\pi}{2}} \frac{dy}{(\sqrt{x^2+\sin^2y})^3}}{4\left(\int_0^{\frac{\pi}{2}} \frac{dy}{\sqrt{x^2+\sin^2y}}\right)^2}.
\end{eqnarray}
This expression shows that, when $\l \pi >1$ there exists a unique ${\tilde{\a}}_0={\tilde{\a}}_0(\l)$ such that $h$ has a global minimum at $x={\tilde{\a}}_0$, defined by \eqref{defalpha0}.
Moreover we have $h({\tilde{\a}}_0)=-\frac{{\tilde{\a}}_0^2}{8}$. Thus, by using \eqref{bbb} we obtain a lower-bound on the energy:
\begin{eqnarray}\label{firstlowerboundalpha0}
F_{\b,\k}(\v) \geq  -\frac{{\tilde{\a}}_0^2}{8\b^2}- \frac{\l (\v(1)-N\pi)}{2\b}.
\end{eqnarray}
When $\l$ does not depend on $\b$, then ${\tilde{\a}}_0$ does not depend on $\b$ either. Note that from the definition of the integer part, we have $ 1-T< NT \leq 1$. We now use \eqref{periodicity} to deduce that $0\leq \v(1)-\v(NT)=\v(1)-N\pi<\pi$. This implies that
\begin{equation}\label{secondlowerboundalpha0}
F_{\b,\k}(\v_\b) \geq -\frac{{\tilde{\a}}_0^2}{8\b^2}+O\left(\frac{1}{\b}\right).
\end{equation}

\textit{Upper bound:} To find a matching upper-bound  we take the solution of
\begin{equation}\label{defsurR}
\left\{
\begin{array}{lcll}
\v''&=&\frac{1}{\b^2}\cos \v \sin \v \ \text{ in } \R^+, \\
\v(0)&=&0, \\
\v'(0)&=&\frac{{\tilde{\a}}_0}{\b},
\end{array}
\right.
\end{equation}
with ${\tilde{\a}}_0$ defined by \eqref{defalpha0}. This solution satisfies
  \begin{equation}\label{integrale1er2}
  \v'(x)^2=\frac{\sin^2 \v(x)+{\tilde{\a}}_0^2}{\b^2} \text{ for all } x \text{ in } [0,1].
  \end{equation}
Let $T$ be defined by $$\frac{T}{2}=\int_0^\frac{\pi}{2} \frac{\b \ dy}{\sqrt{{\tilde{\a}}_0^2+\sin^2y}}.$$ From \eqref{integrale1er2} we have that $\v(\frac{T}{2})=\frac{\pi}{2}$. We need to show that $\frac{T}{2}<1$, that is
\begin{equation}
\int_0^{\frac{\pi}{2}} \frac{dy}{\sqrt{{\tilde{\a}}_0^2+\sin^2y}}< \frac{1}{\b}.
\end{equation}
This last inequality holds if ${\tilde{\a}}_0 > \frac{\pi \b}{2}$ since $\int_0^{\frac{\pi}{2}} \frac{dy}{\sqrt{{\tilde{\a}}_0^2+\sin^2y}}< \int_0^{\frac{\pi}{2}}\frac{1}{{\tilde{\a}}_0}$. But we observe that
\[ \frac{1}{2}\int_0^{\frac{\pi}{2}} \sqrt{ \pi^2 \b^2+4\sin^2 y} \rightarrow 1 \]
as $\b \rightarrow 0$. Thus for $\b$ small enough, $\frac{1}{2}\int_0^{\frac{\pi}{2}} \sqrt{ \pi^2 \b^2+4\sin^2 y}< \l \pi$. This means that ${\tilde{\a}}_0 > \frac{\pi \b}{2}$. From \eqref{integrale1er2}, we have
\begin{equation}\label{majorationT}
 \frac{\pi \b}{\sqrt{1+{\tilde{\a}}_0^2}}\leq T \leq \frac{\pi \b}{{\tilde{\a}}_0}.
\end{equation}
Thanks to \eqref{integrale1er2}, we can also show that $\v$ satisfies the symmetry and periodicity properties of \eqref{Symmetry}. We let $N:=E(\frac{1}{T})$ and we have
\begin{eqnarray}\label{11}
F_{\b,\k}(\v)&=&2N\left[ \int_0^\frac{T}{2}\left(\frac{\v'(x)^2}{8}+\frac{\sin^2\v(x)}{8\b^2}\right)dx-\frac{\l \pi}{4\b}\right] \nonumber \\
  & & \phantom{aaaaaaaaaaaa }+ \int_{NT}^1 \left(\frac{\v'(x)^2}{8}+\frac{\sin^2\v(x)}{8\b^2}\right)dx-\frac{\l}{2\b}\left( \v(1)-N\pi \right) \nonumber \\
&=& 2N \left[ \int_0^\frac{T}{2}\left( \frac{\v'(x)^2}{8}+\frac{\sin^2\v(x)+{\tilde{\a}}_0^2}{8\b^2}\right)-\frac{\l \pi}{4\b} -\frac{{\tilde{\a}}_0^2 T}{16\b^2}  \right] \nonumber \\
& & \phantom{aaaaaaaaaaaa }+\int_{NT}^1 \left(\frac{\v'(x)^2}{8}+\frac{\sin^2\v(x)}{8\b^2}\right)dx-\frac{\l}{2\b}\left( \v(1)-N\pi \right) \nonumber \\
&=& 2N\left[\int_0^\frac{\pi}{2}\frac{\sqrt{\sin^2y+{\tilde{\a}}_0^2}}{4\b}-\frac{\l \pi}{4\b}-\frac{{\tilde{\a}}_0^2T}{16} \right] \nonumber \\
& & \phantom{aaaaaaaaaaaa }+\int_{NT}^1 \left(\frac{\v'(x)^2}{8}+\frac{\sin^2\v(x)}{8\b^2}\right)dx-\frac{\l}{2\b}\left( \v(1)-N\pi \right) \nonumber \\
&=& \frac{-{\tilde{\a}}_0^2 NT  }{8\b^2} +\int_{NT}^1 \left(\frac{\v'(x)^2}{8}+\frac{\sin^2\v(x)}{8\b^2}\right)dx-\frac{\l}{2\b}\left( \v(1)-N\pi \right).
\end{eqnarray}
Now we note that \begin{equation}\label{DLdeN}
{N=E(\frac{1}{T})=\frac{1}{T}(1+O_\b(T))}
 \end{equation}because $T \rightarrow 0$ as $\b \rightarrow 0$ from \eqref{majorationT}. From the periodicity property of $\v$, we also have that
 \begin{eqnarray}\label{22}
 \int_{NT}^1 \left(\frac{\v'(x)^2}{8}+\frac{\sin^2\v(x)}{8\b^2}\right)dx &\leq& 2\int_0^\frac{T}{2} \left(\frac{\v'(x)^2}{8}+\frac{\sin^2\v(x)}{8\b^2}\right)dx \nonumber \\
 &\leq & 2\int_0^\frac{T}{2} \left(\frac{\v'(x)^2}{8}+\frac{\sin^2\v(x)+{\tilde{\a}}_0^2}{8\b^2}\right)dx-\frac{{\tilde{\a}}_0^2 T}{8\b^2} \nonumber \\
 & \leq & 2\int_0^\frac{\pi}{2} \frac{\sqrt{\sin^2 y+{\tilde{\a}}_0^2}\ dy}{4\b}-\frac{{\tilde{\a}}_0^2 T}{8\b^2} \nonumber \\
 & \leq & \frac{\l \pi}{2\b}-\frac{{\tilde{\a}}_0^2 T}{8\b^2}.
 \end{eqnarray}
 In the last inequality we have used the definition of ${\tilde{\a}}_0$ \eqref{defalpha0}. By using \eqref{majorationT} and \eqref{22} we find that \begin{equation}\label{33}
  \int_{NT}^1 (\frac{\v'(x)^2}{8}+\frac{\sin^2\v(x)}{8\b^2})dx=O(\frac{\l}{\b})=O(\frac{1}{\b}).
 \end{equation}
 We also recall from the periodicity property  $\v(x+T)=\pi+\v(x)$ for $x$ in $[0,1-T]$, that $0\leq \v(1)-N\pi<\pi$. We then conclude from \eqref{22} and \eqref{33} that
 \begin{equation}\label{upperboundalpha0}
 F_{\b,\k}(\v) \leq \frac{-{\tilde{\a}}_0^2   }{8\b^2}+O\left(\frac{1}{\b}\right).
 \end{equation}

\end{proof}

\begin{lemma}\label{enegative}
Let $\l>\frac{1}{\pi}$, let $e$ be defined by \eqref{defe} and $\mathcal{F}$ be defined by \eqref{defF}, then we have
\begin{equation}\label{eenegative}
e \leq \frac{1-\l \pi}{4\b}(1+o_\b(1))<0,
\end{equation}
for $\b$ small enough.
\end{lemma}

\begin{proof}
We construct a test function which proves \eqref{eenegative}. Let $0<\gamma<\frac{T}{2}$ to be fixed later. We take $u$ the solution of $u'=\sin u$ on $(\gamma,\frac{T}{2})$ such that $u(\frac{T}{2})=\frac{\pi}{2}$. That is $u(x)=2\arctan (e^{x-\frac{T}{2}})$. We set $\eta =2\arctan(e^{\gamma-\frac{T}{2}})$ and in $[0,\gamma]$ we take $u(x)=\frac{\eta x}{\gamma}$. We thus have
\begin{eqnarray}
\int_0^{\frac{T}{2}}\left(\frac{u'(x)^2}{8}+\frac{\sin^2 u(x)}{8\b^2}\right)dx&=& \int_0^\gamma \frac{\eta^2}{8\gamma^2}dx+\frac{1}{8\beta^2}\int_0^\gamma \sin^2(\frac{\eta x}{\gamma})dx  \nonumber \\
& &\phantom{aaaaaaaa}+\int_\gamma^\frac{T}{2}\frac{|u'(x)||\sin u(x)|}{4\b}dx \nonumber \\
&\leq & \frac{\eta^2}{8\gamma}+\frac{\gamma}{\beta^2}+\int_\eta^\frac{\pi}{2}\frac{|\sin y|}{4\b}\ dy \nonumber \\
&\leq & \frac{\eta^2}{8\gamma}+\frac{\gamma}{\beta^2}+\frac{\cos \gamma}{4\b}, \nonumber
\end{eqnarray}
We then choose $\gamma =\beta^3$ (note that $\beta^3<T$ for $\beta$ small since $T=2\beta \int_0^\frac{\pi}{2}\frac{1}{\sqrt{\sin^2y+{\tilde{\a}}_0^2}}\geq \frac{\pi \beta}{2\sqrt{1+4\l^2}}$ because ${\tilde{\a}}_0\leq 2\l$) and this yields the result.
\end{proof}

\begin{proposition}\label{evaluationdephi1}
Let $\k=\frac{\l}{\b}$, with $\l>\frac{1}{\pi}$ independent of $\b$. Let $\v_\b$ be a minimizer of $F_{\b,\k}$ in $\J$. We set $\psi_\b(x):=\b \v_\b(x)$, then up to a subsequence, there exists $\psi_0$ in $H^1((0,1))$ such that $\psi_\b \rightharpoonup \psi_0$ in $H^1((0,1))$ and $\psi_\b \rightarrow \psi_0$ in $\mathcal{C}^0([0,1])$. Furthermore, there exists $l>0$ such that $\lim_{\b\rightarrow 0} \b \v_\b(1)=l$. In particular, if $N:=E(\frac{\v_\b(1)}{\pi})$, then there exist $0<c<C$ such that
\begin{equation}
\frac{c}{\b}\leq N \leq \frac{C}{\b}.
\end{equation}
\end{proposition}

\begin{proof}
By using \eqref{DLenergie2} we find that
\begin{eqnarray}
F_{\b,\k}(\v_\b)= \frac{1}{8}\int_0^1 \left((\v_\b'(x)-2\k)^2+\frac{1}{\b^2}\sin^2\v_\b(x)\right)dx -\frac{\k^2}{2}= \frac{-{\tilde{\a}}_0^2}{8\b^2}(1+o_\b(1)). \nonumber
\end{eqnarray}
We now use that $\psi_\b'(x)^2=\b^2\v_\b'(x)^2$ and $\k=\b \l$ to obtain that
\begin{equation}
\int_0^1 \left(\psi_\b'-2\l \right)^2dx \leq C,
\end{equation}
for some constant $C$ which does not depend on $\b$. Since $\psi_\b(0)=0$, this implies that $\psi_\b$ is bounded in $H^1((0,1))$. In particular, up to a subsequence, there exists $\psi_0$ in $H^1((0,1))$ such that $\psi_\b \rightharpoonup \psi_0$ in $H^1((0,1))$ and $\psi_\b \rightarrow \psi_0$ in $\mathcal{C}^0([0,1])$.
We call $N:=E\left(\frac{\v_\b(1)}{\pi} \right)$ the number of periods. It follows from Proposition \ref{propcask>1surpi} that we have $F_{\b,\k}(\v_\b)=\frac{ -{\tilde{\a}}_0^2}{8\b}(1+o_\b(1))$. By using \eqref{firstlowerbound} we conclude that $N\geq \frac{c}{\b}$ for some $c>0$.
  And this, along with the uniform convergence of $\psi_\b$ implies that $\lim_{\b\rightarrow 0} \b \v_\b(1)=l$, for some $l>0$. In particular this implies that $N$ satisfies $\frac{c}{\b}\leq N \leq \frac{C}{\b}$ for some constants $0<c<C$.
\end{proof}

\begin{proposition}
Let $\l=\k \b >\frac{1}{\pi}$ with $\l$ independent of $\b$. Let $\v_\b$ be a minimizer of $F_{\b,\k}$ in $\mathcal{J}$. Let $\tilde{\v}_\b(x):= \v_\b(\b x)$ defined in $[0,\frac{1}{\b}]$. We have that $\tilde{\v}_\b$ converges in $C^\infty(K)$ for every $K \subset [0,+\infty)$ towards $\v_0$ a solution of \eqref{eqlimitk>1/pi}.
\end{proposition}

\begin{proof}
We let $\a_\b:=\v'_\b(0)$. The function $\tilde{\v}_\b$ satisfies
\begin{equation}
    \left\{
    \begin{array}{lcll}
    \tilde{\v_\b}''&=&\sin \tilde{\v}_\b \cos \tilde{\v}_\b \text{ in }
    (0,\frac{1}{\b}), \\
    \tilde{\v}_\b(0)&=&0, \\
    \tilde{\v}_\b'(0)&=& \b \a_\b.
    \end{array}
    \right.
    \end{equation}
Let us recall from Proposition \ref{Symmetry} that $\v_\b$ is quasi-periodic. We call $N=E\left(\frac{\v_\b(1)}{\pi} \right)$, and we have from the previous proposition that $\frac{c}{\b}\leq N \leq \frac{C}{\b}$. Let $K=[m;M]$ be a compact subset of $[0,\frac{1}{\b})$. On the interval $\b K\subset [0,1)$ there is a finite number of periods of $\v_\b$, let us call $L$ this number. Since the energy per period of $\v_\b$ is of order $\frac{1}{\b}$ (because the total energy is of order $\frac{1}{\b^2}$) we have that
\[ \int_{\b K} \left( \v_\b'(x)^2+\frac{1}{\b^2}\sin^2\v_\b(x)\right)dx-\frac{\l}{2\b}\v_\b(M\b)=\frac{-A}{\b}(1+o_\b(1)) \]
for some $A>0$. Since there {are exactly L} periods on $\b K$, we have $L\pi\leq \v_\b(M\b)<(L+1)\pi$. This yields that
\[ \int_{\b K} \left( \v_\b'(x)^2+\frac{1}{\b^2}\sin^2\v_\b(x)\right)dx \leq \frac{B}{\b}\]
for some $B>0$. But
\[\int_{\b K} \left( \v_\b'(x)^2+\frac{\sin^2\v_\b(x)}{\b^2}\right)dx= \frac{1}{\b}\int_K \left(\tilde{\v}_\b'(x)^2+ \sin^2\tilde{\v}_\b(x)\right)dx.\]
Hence we obtain that $\tilde{\v}_\b$ is bounded in $H^1_{\text{loc}}(\R^+)$ and converges weakly in that space to some $\v_0$. From the weak convergence in $H^1_{\text{loc}}(\R^+)$ and the strong convergence in $\mathcal{C}^0_{\text{loc}}([0,+\infty))$ we obtain that $\v_0$ satisfies $\v_0''=\sin \v_0 \cos \v_0$ in $\R^+$. From the regularity theory, $\v_0$ is in $\mathcal{C}^\infty(\R^+)$. We also have $\v_0(0)=0$ from the uniform convergence in compact sets of $[0,+\infty)$.
We also set $\tilde{\a}_\b = \frac{\a_\b}{\b}$. Because of the minimizing property of $\v_\b$, and from \eqref{DLenergie2} and \eqref{defh} we have that
\begin{equation}
h({\tilde{\a}}_0) \leq h(\tilde{\a}_\b) \leq h({\tilde{\a}}_0)(1+o_\b(1)),
\end{equation}
where $h$ is defined by \eqref{defh2}, ${\tilde{\a}}_0$ is the minimizer of $h$ and satisfies $h({\tilde{\a}}_0)= \frac{-{\tilde{\a}}_0^2}{8}$. Since ${\tilde{\a}}_0$ is the unique minimizer of $h$ this implies that
\begin{equation}
\tilde{\a}_\b \rightarrow {\tilde{\a}}_0, \ \ \text{ as } \b \rightarrow 0.
\end{equation}
By a bootstrap argument, we can show that $\tilde{\v}_\b$ converges in $\mathcal{C}^\infty_{\text{loc}}([0,+\infty))$ and hence satisfies \eqref{eqlimitk>1/pi}.
\end{proof}

When $\l$ tends to $+\infty$ as $\b$ goes to $0$, we can get the expansion of the minimum of the energy in a simpler way since $\v$ is almost linear.
\begin{proposition}\label{caskgrand}
Let $\l=\k \b > \frac{1}{\pi}$, let $\v_\b$ be a minimizer of $F_{\b,\k}$ in $\mathcal{J}$. Then if $\l \rightarrow +\infty$ as $\b \rightarrow 0$ we have
\begin{equation}\label{DL3}
F_{\b,\k}( \v_\b)= -\frac{\l^2}{2\b^2}(1+o_\b(1))
\end{equation}
and
\begin{equation}
\|\frac{\v_\b(x)}{2\k x}-1\|_{L^\infty((0,1))}\rightarrow 0 \text{ and } \|\v_\b'(x)-2\k\|_{L^\infty((0,1))} \rightarrow 0,
\end{equation}
as $\b$ tends to zero.
\end{proposition}

\begin{proof}
 Taking $x=1$ in \eqref{integrale1eree} yields $4\k^2=\frac{1}{\b^2}\sin^2\v(1)+\v'(0)^2$. We thus find that $\v'(x)^2 \geq \v'(0)^2=4\k^2-\frac{\sin^2\v(1)}{\b^2}$ and then $\v'(x)\geq \frac{\sqrt{4\l^2-1}}{\b}$ for every $x$ in $[0,1]$. We also have that $\v'(x)^2 \leq \frac{1}{\b^2}+\v'(0)^2$ and $\v'(0) \leq 2\k$. Thus for every $x$ in $[0,1]$
\begin{equation}\label{derive}
\frac{\sqrt{4\l^2-1}}{\b} \leq \v'(x) \leq \frac{\sqrt{4\l^2+1}}{\b}.
\end{equation}
In particular we have $\v(1) \leq \frac{\sqrt{4\l^2+1}}{\b}$. We can then say that
\[ F_{\b,\k}(v_\b) \geq \frac{4 \l^2 -1}{8\b^2}- \frac{\l}{2\b}\frac{\sqrt{4\l^2+1}}{\b}. \]
This implies that $F_{\b,\k}(v_\b) \geq -\frac{\l^2}{2\b^2}(1+o_\b(1))$ if $\l(\b)\rightarrow +\infty$ as $\b \rightarrow 0$. The upper-bound is obtained with the test function: $\v(x)=2\k x$ and it yields \eqref{DL3}. It also follows from \eqref{derive} that $\|\frac{\v_\b(x)}{2\k x}-1\|_{L^\infty((0,1))}\rightarrow 0 \text{ and } \|\v_\b'(x)-2\k\|_{L^\infty((0,1))} \rightarrow 0,$ as $\b$ goes to zero.
\end{proof}

The expansion \eqref{DL3} is in agreement with \eqref{DLenergie2}, since when $\l$ is large,  $\l \sim 2{\tilde{\a}}_0$. Theorem  \ref{casksup1/pi} follows from the Propositions of this section.

\subsection{The intermediate case $\k\pi=1+\eta$ for $\eta$ small}
In this short section, we study the intermediate regime $\l \pi=1+\eta_\b$ for some small $\eta_\b$ as $\b\rightarrow 0$.
\begin{lemma}\label{lemalpha0} As ${\tilde{\a}}_0\to 0$, then
\begin{equation}\label{DLalpha0}
\int_0^{\frac{\pi}{2}} \sqrt{{\tilde{\a}}_0^2+\sin^2y}\ dy -1 \sim \frac{{\tilde {\alpha}_0}^{2}}{2} \log \frac{1}{{\tilde{\a}}_0}.
\end{equation}
\end{lemma}
\begin{proof} We have \begin{eqnarray}
\int_0^{\frac{\pi}{2}} \sqrt{{\tilde{\a}}_0^2+\sin^2y}\ dy-1 &=& \int_0^{\frac{\pi}{2}} \sqrt{{\tilde{\a}}_0^2+\sin^2y}\ dy- \int_0^{\frac{\pi}{2}}\sin y\ dy \nonumber \\
&=& \int_0^{\frac{\pi}{2}} \frac{{\tilde{\a}}_0^2\ dy}{\sin y +\sqrt{{\tilde{\a}}_0^2+\sin^2y}} \nonumber.
\end{eqnarray}
We are going to prove that $ \int_0^{\frac{\pi}{2}} \frac{dy}{\sin y +\sqrt{{\tilde{\a}}_0^2+\sin^2y}} \sim  \frac{1}{2} \log \frac{1}{{\tilde{\a}}_0}$, which will conclude the proof. Indeed,
\[\int_0^{\frac{\pi}{2}} \frac{dy}{\sin y +\sqrt{{\tilde{\a}}_0^2+\sin^2y}}= \int_0^{{\tilde{\a}}_0} \frac{dy}{\sin y +\sqrt{{\tilde{\a}}_0^2+\sin^2y}}+\int_{{\tilde{\a}}_0}^\frac{\pi}{2}\frac{dy}{\sin y +\sqrt{{\tilde{\a}}_0^2+\sin^2y}}. \]
Since $\sin y+\sqrt{{\tilde{\a}}_0^2+\sin^2y}\geq {\tilde{\a}}_0$, we deduce that
\[ \int_0^{{\tilde{\a}}_0} \frac{dy}{\sin y +\sqrt{{\tilde{\a}}_0^2+\sin^2y}}\leq 1.\]
On the other hand $\sqrt{{\tilde{\a}}_0^2+\sin^2y}\leq {\tilde{\a}}_0+\sin y \leq {\tilde{\a}}_0+y$. Thus $\sin y +\sqrt{{\tilde{\a}}_0^2+\sin^2y}\leq 2y+{\tilde{\a}}_0$ and
\begin{eqnarray}\label{equiv1}
\int_{{\tilde{\a}}_0}^\frac{\pi}{2} \frac{dy}{\sin y +\sqrt{{\tilde{\a}}_0^2+\sin^2y}} &\geq& \frac 12 \int_{{\tilde{\a}}_0}^\frac{\pi}{2} \frac{dy}{y+{\tilde{\a}}_0/2} \nonumber \\
& \geq & \frac{1}{2}\left( \log(\frac{\pi}{2}+\frac{{\tilde{\a}}_0}{2})-\log\frac{3 {\tilde{\a}}_0}{2} \right) \nonumber \\
& \geq &  \frac{1}{2}\log \frac{1}{{\tilde{\a}}_0}-\frac{1}{2}\log \frac{3}{2}.
\end{eqnarray}
We also have that $\sin y +\sqrt{{\tilde{\a}}_0^2+\sin^2y}\geq 2\sin y$ and
\begin{eqnarray}\label{equiv2}
\int_{{\tilde{\a}}_0}^\frac{\pi}{2} \frac{dy}{\sin y +\sqrt{{\tilde{\a}}_0^2+\sin^2y}} &\leq& \int_{{\tilde{\a}}_0}^\frac{\pi}{2} \frac{dy}{\sin y} \nonumber \\
& \leq & \frac{1}{2}\left( \log \tan \frac{\pi}{4}-\log \tan \frac{{\tilde{\a}}_0}{2}\right) \nonumber \\
& \leq & -\frac{1}{2}\log \tan \frac{{\tilde{\a}}_0}{2}
\end{eqnarray}
 Therefore, \eqref{equiv1} and \eqref{equiv2} yield $ \int_0^{\frac{\pi}{2}} \frac{dy}{\sin y +\sqrt{{\tilde{\a}}_0^2+\sin^2y}} \sim  \frac{1}{2} \log \frac{1}{{\tilde{\a}}_0}$.
\end{proof}

\begin{proposition}\label{prop1surpi}
Let us assume that $\l \pi =1 +\frac{\b^{2\g}}{2} \log \frac{1}{\b^\g}$ for $0< \g< 1$. Then a minimizer of $F_{\b,\k}$ in $\mathcal{J}$ satisfies
\begin{equation}\label{alphbet}
F_{\b,\k}(\v_\b)= -\frac{1}{8\b^{2-\g}}(1+o_\b(1)).
\end{equation}
\end{proposition}

Note that when $\g$ goes to zero, then $\l$ tends to $\frac{1}{\pi}$ and we find that the energy blows up as $\frac{-1}{\b^2}$ which is the behaviour described in Theorem \ref{casksup1/pi}. When $\g$ tends to $1$ then we recover that the energy blows up as $\frac{1}{\b}$, which is the same behaviour as the one observed for $\l<\frac{1}{\pi}$ with $\l$ independent of $\b$.
\begin{proof}
Let ${\tilde{\a}}_0$ be the number defined by \eqref{defalpha0}. Since $\l$ goes to $1/\pi$, ${\tilde{\a}}_0$ goes to $0$ and
\begin{equation}\label{DL6}
\int_0^{\frac{\pi}{2}} \sqrt{{\tilde{\a}}_0^2+\sin^2y}\ dy= \l \pi= 1 +\frac{\b^{2\g}}{2} \log \frac{1}{\b^\g}.
\end{equation} Therefore, from Lemma \ref{lemalpha0}, we find
\begin{equation}\label{equivalent}
\frac{{{\tilde{\a}}_0}^2}{2}  \log \frac{1}{{\tilde{\a}}_0}\sim \frac{\b^{2\g}}{2} \log \frac{1}{\b^\g}.
\end{equation}
A simple computation yields \begin{equation}\label{equivalentdealpha0}
{\tilde{\a}}_0 \sim \b^{\g},  \text{as } \b \text{ goes to } 0.
\end{equation}
 Coming back to \eqref{firstlowerboundalpha0} in the proof of Proposition \ref{propcask>1surpi},  we see that
\[F_{\b,\k}(\v_\b)\geq \frac{-{\tilde{\a}}_0^2}{8\b^2}-\frac{\l(\v_\b(1)-N\pi)}{2\b} \]
with $N=E(\frac{\v_\b(1)}{\pi})$. From the equivalent of ${\tilde{\a}}_0$ of \eqref{equivalentdealpha0} we deduce
\begin{equation}
F_{\b,\k}(\v_\b)\geq \frac{-1}{8\b^{2-\gamma}}(1+o_\b(1)).
\end{equation}
To find an upper-bound we take the solution of $\v''=\frac{1}{\b^2}\cos \v \sin \v$ such that $\v(0)=0$ and $\v'(0)=\frac{{\tilde{\a}}_0}{\b}$. The main difference with the proof of Theorem \ref{casksup1/pi} is that now we do not know that this solution reaches $\frac{\pi}{2}$. If it does, then we can adapt the proof and in particular  \eqref{majorationT}, using that the period $T$ satisfies $T \leq C \b^{1-\g}$. Hence we find that $F_{\b,\k}(\v_\b)\leq \frac{-{\tilde{\a}}_0^2}{\b^2}(1+o_\b(1))$ as in \eqref{upperboundalpha0}. Now if the solution does not reach $\frac{\pi}{2}$ then we can still argue that:
\begin{eqnarray}
F_{\b,\k}(\v)&=& \int_0^1 \left( \frac{\v'(x)^2}{8}+\frac{\sin^2\v(x)}{8\b^2} \right)dx -\frac{\l}{2\b}\v(1) \nonumber \\
&=& \int_0^1 \left( \frac{\v'(x)^2}{8}+\frac{\sin^2\v(x)+{\tilde{\a}}_0^2}{8\b^2} \right)dx-\frac{\l}{2\b}\v(1)-\frac{{\tilde{\a}}_0^2}{8\b^2} \nonumber \\
&=& \int_0^{\v(1)} \frac{\sqrt{\sin^2y+{\tilde{\a}}_0^2}}{4\b}\ dy-\frac{\l}{2\b}\v(1)-\frac{{\tilde{\a}}_0^2}{8\b^2} \nonumber \\
&\leq & \frac{\pi }{4\b}-\frac{\l}{2\b}\v(1)-\frac{{\tilde{\a}}_0^2}{8\b^2}\nonumber \\
& \leq & -\frac{{\tilde{\a}}_0^2}{8\b^2}(1+o_\b(1))
\end{eqnarray}
where in the last equality we have used \eqref{equivalentdealpha0}  and this yields \eqref{alphbet}.
\end{proof}

%
%
%
%

\section{The full energy $G_{\e,\d,\k}(v,\v)$}\label{2}

\subsection{First properties of minimizers of $G_{\e,\d,\k}(v,\v)$}
The aim of this section is to prove that if $(v_\e,\v_\e)$ is a minimizer of $G_{\e,\d,\k}$, then $v_\e$ is close to 1 and $\v_\e$ is an almost minimizer of $F_{\e/\sqrt\d, \k}$.

\begin{proposition}For every $\e,\d,\k>0$, there exists a minimizer of $G_{\e,\d,\k}(v,\v)$ in $\I$.
 It satisfies the following Euler-Lagrange equations
\begin{equation}\label{EulerLagrange}
\left\{
\begin{array}{lcll}
-v''+\frac{1}{\e^2}v(v^2-1)+\frac{1}{4}v\v'^2+\frac{\d}{2\e^2}v^3\sin^2\v-\k v\v'=\lambda v \text{ in } (0,1), \\
-(v^2\v')'+\frac{\d}{\e^2}v^4\cos \v \sin \v +2\k(v^2)'=0 \text{ in } (0,1),
\end{array}
\right.
\end{equation}
\begin{equation}\label{EulerLagrange0}
\left\{
\begin{array}{lcll}
v'(0)=v'(1)=0, \\
\v(0)=0, \ \ \v'(1)=2\k,
\end{array}
\right.
\end{equation}
where $\lambda$ is a Lagrange multiplier. Besides $(v,\v)$ is smooth in $[0,1]\times [0,1] \setminus \mathcal{S}(v)$, $v$ is nonnegative  and $(v,\v)$ satisfies the relation
\begin{equation}\label{Integrale1ere}
v'^2+\frac{v^2\v'^2}{4}-\frac{1}{\e^2}(\frac{v^4}{2}-v^2)-\frac{\lambda^2}{2}v^2-\frac{\d}{2\e^2}v^4\sin^2\v = const \ \ \text{in } [0,1].
\end{equation}
\end{proposition}

\begin{proof} The existence of minimizers is not direct because for a minimizing sequence $(v_n,\v_n)$ in $\I$, $\v_n'^2$ could be unbounded in $L^2((0,1))$ near the points of $\mathcal{S}(v)$. However, one can adapt the argument of \cite{AmbrosioVirga91} to prove the existence of minimizers.

The fact that $(v,\v)$ satisfies the Euler-Lagrange equations is classical. Note that the boundary condition  is: $v(1)^2(\v'(1)-2\k)=0$. However, $v(1)\neq 0$ since  otherwise, $v'(1)=0$ would imply that $v$ is identically zero and this contradicts $\min_{\I}G_{\e,\d,\k} \leq F_{\b,\k}(\v_\b)<0$.

We note that if $(v_\e,\v_\e)$ is a minimizer  then $(|v_\e|,\v_\e)$ is also a minimizer, thus we can assume $v_\e>0$.

 The fact that $(v,\v)$ satisfies relation \eqref{Integrale1ere} is obtained by multiplying the first equation of \eqref{EulerLagrange} by $v'$ and the second equation of \eqref{EulerLagrange} by $\v'$, adding the two equations  and integrating.
\end{proof}
In order to use the results about the simplified functional $F_{\b,\k}$, with $\b=\frac{\e}{\dd}$ we first prove that  $v_\e$ converges to 1. We will always assume that $\e^2=o_\e(\d)$ since we want the parameter $\b$ to be small.

\begin{proposition}\label{convv1}
Let $(v_\e,\v_\e)$ be a minimizer of $G_{\e,\d,\k}$ in $\I$. Then for $\e>0$ small enough, there exists $C>0$ independent of the parameters, such that
\begin{equation}\label{majjv-1}
\|1-v_\e\|_{L^\infty((0,1))} \leq C\k \sqrt{\e}.
\end{equation}
\end{proposition}

\begin{proof} Since $G_{\e,\d,\k}(1,0)=0$ we find that
\begin{eqnarray}\label{ecriturepositive}
G_{\e,\d,\k}(v_\e,\v_\e)=\frac12 \int_{0}^1 v_\e'^2 + \frac{1}{4\e^2}\int_{0}^1 (1-v_\e^2)^2 +\frac{1}{8} \int_{0}^1 v_\e^2(\v_\e'-2\k)^2 \nonumber \\
+ \frac{\d}{8\e^2}\int_{0}^1v_\e^4\sin^2\v_\e -\frac{\k^2}{2}\leq 0.
\end{eqnarray}
Therefore,
\begin{equation}\label{1eremajorationsurv}
\frac 12 \int_{0}^1 v_\e'^2+\frac{1}{4\e^2}\int_{0}^1(v_\e^2-1)^2 \leq \frac{\k^2}{2} .
\end{equation}
We now apply Modica-Mortola's trick and the coarea formula (cf.\ \cite{EvansGariepy2015})to obtain
\begin{eqnarray}
\frac{1}{2\sqrt{2}\e}\int_{0}^1 |v_\e'(x)||1-v_\e(x)^2|dx &\leq& \frac{\k^2}{2} \nonumber \\
\Rightarrow \frac{1}{2\sqrt{2}\e}\int_{-\infty}^\infty |1-t^2| \mathcal{H}^0(v_\e=t)dt &\leq& \frac{\k^2}{2} \nonumber \\
\Rightarrow \int_{\inf v_\e}^{\sup v_\e} |1-t^2|dt &\leq & \sqrt{2}\k^2 \e \nonumber
\end{eqnarray}
where for the last inequality we have used that $\mathcal{H}^0(v_\e=t)\geq 1$ if $ \inf v_\e \leq t \leq \sup v_\e$ and  $\mathcal{H}^0(v_\e=t)=0$ everywhere else. We first observe that $0\leq \inf v_\e\leq 1$ (because we assume $v_\e\geq 0$ and $\int_{0}^1 v^2=1$).
\item[a)] If $\sup v_\e>1$:  we then have  \\
 \[\int_{\inf v_\e}^1 (1-t^2)dt + \int_1^{\sup v_\e}  (t^2-1)dt \leq \sqrt{2}\e \k^2\] which implies that
\[ \int_{\inf v_\e}^1 (1-t^2)dt \leq \sqrt{2} \k^2 \e \ \text{ and } \int_{1}^{\sup v_\e} (t^2-1)dt \leq \sqrt{2} \k^2 \e .\]
We set $m=\inf v_\e$ and $M=\sup v_\e$  and we obtain
\begin{eqnarray}
\sqrt{2}\k^2\e & \geq & 1-m-\frac{1-m^3}{3} \geq  (1-m)(1-\frac{1+m+m^2}{3}) \nonumber \\
& \geq & (1-m)\frac{2-m-m^2}{3}. \nonumber
\end{eqnarray}
We have $m^2\leq m$ and hence we find that
\begin{eqnarray}
\frac{2}{3}(1-m)^2 \leq  \sqrt{2}\k^2 \e \text{   and   }
(1-m)\leq  C \k \sqrt{\e}. \nonumber
\end{eqnarray}
We also have
\begin{eqnarray}
\sqrt{2}\k^2 \e &\geq & \frac{M^3-1}{3}-(M-1)
 \geq  (M-1) \frac{M^2+M-2}{3} \nonumber \\
& \geq & \frac{2}{3}(M-1)^2.
\end{eqnarray}
It follows that $(M-1) \leq C\k \sqrt{\e}$ and then $(M-m) \leq C\k \sqrt{\e}$. This yields \eqref{majjv-1}.
\item[b)] If $\sup v_\e \leq 1$: we have
\begin{equation}
\sqrt{2}\k^2 \e \geq  \int_m^M(1-t^2)dt \geq  M-m- \frac{M^3-m^3}{3}  \geq  (M-m)(1-\frac{M^2+mM+m^2}{3}). \nonumber
\end{equation}
Since $M^2+mM+m^2 \leq 1+2m$ we have $1-\frac{M^2+mM+m^2}{3}\geq \frac{2}{3}(1-m)$. Hence
\begin{eqnarray}
\sqrt{2}\k^2 \e  \geq   (M-m)(1-m) \geq  (M-m)^2. \nonumber
\end{eqnarray}
We deduce that $(M-m) \leq C\k \sqrt{\e}$. On the other hand we have $$2\k^2\e^2\geq \int_{0}^1 (1-v^2)^2 \geq \int_{0}^1 (1-M^2)^2.$$ Thus $1-M^2\leq C\k \e$ for some $C>0$ and $1-M \leq C \k \e$. This proves that \eqref{majjv-1} holds.
\end{proof}

The previous proposition shows that for $\k \sqrt{\e}$ small enough, a minimizer $(v_\e,\v_\e)$ of $G_{\e,\d,\k}$ in $\I$ is not only in $\I$ but also in $H^1((0,1)) \times H^1((0,1))$ and then we can prove that it is smooth everywhere in $[0,1]$. From now on we will always assume that $\k \sqrt{\e}$ goes to 0.

\begin{proposition}\label{convv2}
Let $(v_\e, \v_\e)$ be a minimizer of $G_{\e,\d,\k}$ in $\mathcal{I}$.
\begin{itemize}
\item[1)] If $\k$ is bounded independently of $\e$, and $\d$ is also independent of $\e$ then \eqref{93} and \eqref{15} hold.
\item[2)] If $\k=\frac{\l \dd}{\e}$, with $\l<\frac{1}{\pi}$ independent of $\e$, $\e^2=o_\e(\d)$ and $\d=o_\e(\e)$ then \eqref{94} holds and $\| \v_\e \|_{L^\infty((0,1))}\leq \v_\e(1)(1+o_\ep(1))$.

\end{itemize}
Moreover, in both cases,
\begin{equation}\label{estGG} G_{\e,\d,\k}(v_\e,\v_\e)=G_{\e,\d,\k}(1,\v_\e)(1+o_\e(1))=\inf F_{ \ep/{\sqrt\delta},\k}(1+o_\e(1)) .\end{equation}
\end{proposition}

\begin{proof}
 Let $(v,\v)$ be a minimizer of $G_{\e,\d,\k}$ in $\mathcal{I}$ and let $\tilde{\v}_\e$ be a minimizer of $F_{\b,\k}$ in $\mathcal{J}$, with $\beta=\ep/\sqrt\delta$. We have
 \begin{equation} \label{upperbG}G_{\e,\d,\k}(v,\v)\leq G_{\e,\d,\k}(1,\tilde{\v}_\e)=F_{\b,\k}(\tilde{\v}_\e)\leq F_{\b,\k}(\v). \end{equation}

 \item[1)] If $\k$ is bounded we use \eqref{DL22} to get
 \begin{equation}\label{9}
 \frac{1}{8}\int_0^1v^2\v'^2+\frac{\d}{8\e^2}\int_0^1v^4\sin^2\v-\frac{\k}{2}\int_0^1v^2\v'\leq \frac{-\k^2 \e}{2\sqrt{\d}}(1+o_\e(1)).
 \end{equation}
We write
\begin{equation}\label{99}
\frac{\k}{2}\int_0^1v^2\v'=\frac{\k}{2}\int_0^1\v'+\frac{\k}{2}\int_0^1(v^2-1)\v'.
\end{equation}
Besides
\begin{equation}\label{eqestv}
 \left| \frac{\k}{2}\int_0^1 (v^2-1)\v' \right|
  \leq  \frac{\k}{4}\left[ \e^{2\a} \int_0^1 \v'^2+\frac{1}{\e^{2\a}}\int_0^1(v^2-1)^2 \right]
\end{equation}
 for some $\a>0$ to be chosen later. We recall from \eqref{1eremajorationsurv} that $\int_0^1(v^2-1)^2\leq 2\k^2\e^2$. From \eqref{9},  \eqref{99} and \eqref{eqestv}, we obtain
 \begin{eqnarray}
  \frac{1}{8}\int_0^1v^2\v'^2+\frac{\d}{8\e^2}\int_0^1v^4\sin^2\v -\frac{\k}{2}\v(1) \phantom{aaaaaaaaaaa} \nonumber \\
  \phantom{aaaaaaaaaaaaaa}-\frac{\k}{4} \left( \e^{2\a}\int_0^1 \v'^2+2\k^2\e^{2(1-\a)} \right)\leq \frac{-\k^2 \e}{\dd}(1+o_\e(1)).
 \end{eqnarray}
  We recall from Proposition \ref{convv1} that $\|v^2-1\|_{L^\infty((0,1))}\leq C\k \sqrt{\e}$. We chose $\a>0$ such that $\e^{2(1-\a)}=o_\e(\e)$. For instance, we can take $\a=\frac{1}{4}$ and we obtain

 \begin{eqnarray}
 \left( \frac{1}{8} \int_0^1\v'^2+\frac{\d}{8\e^2}\int_0^1\sin^2\v \right)(1+O_\e(\sqrt{\e}))-\frac{\k}{2}\v(1) &\leq& \frac{-\k^2 \e}{2\sqrt{\d}}(1+o_\e(1)) \nonumber \\
 \frac{\dd}{4\e}\int_0^1|\v'||\sin \v| (1+O_\e(\sqrt{\e})) -\frac{\k}{2}\v(1) &\leq& \frac{-\k^2 \e}{2\sqrt{\d}}(1+o_\e(1)). \label{eqmm}
 \end{eqnarray} This implies that $\v(1)>0$.
  We let $N:=E \left(\frac{\v(1)}{\pi} \right)$ and we deduce
 \begin{equation}\nonumber
 \frac{\dd}{\e}\left(\frac{N}{2} +\frac{1-\cos(\v(1)-N\pi)}{4}\right)(1+O_\e(\sqrt{\e}))-\frac{\k}{2}\v(1) \leq \frac{-\k^2 \e}{2\dd}(1+o_\e(1)).
\end{equation}
 We can rewrite this last inequality as
 \begin{eqnarray}
 \frac{N\dd}{2\e}(1-\frac{\k\e\pi}{\dd})+\frac{\dd}{\e}f(\v(1)-N\pi) \phantom{aaaaaaaaaaaaaaaaaaaaaaaaaa}\nonumber \\
 \phantom{aaaa} +\frac{\dd}{\e}\left( \frac{N}{2}+\frac{1-\cos(\v(1)-N\pi)}{4} \right)\times O_\e(\sqrt{\e})\leq \frac{-\k^2 \e}{2\dd}(1+o_\e(1)) \label{eqN}
 \end{eqnarray}
 with $f$ defined by \eqref{deff}. From the study of the function $f$ done in Proposition \ref{DLenergie1} we have $f(\v(1)-N\pi) \geq \frac{-\k^2 \e^2}{2\d}(1+o_\e(1))$ thus
 \begin{eqnarray}
 \frac{N\dd}{2\e}(1-\frac{\k\e\pi}{\dd})+\frac{\dd}{\e}\left( \frac{N}{2}+\frac{1-\cos\left(\v(1)-N\pi\right)}{4} \right)\times O_\e(\sqrt{\e}) \leq o_\e(\e). \nonumber
 \end{eqnarray}
 This implies that $N=0$  for $\e$ small enough. We now come back to inequality \eqref{9} and we recall that we only used the fact that $\int_0^1\frac{v'^2}{2}+\frac{1}{4\e^2}(1-v^2)^2 \geq 0$. Keeping track of this term in the computations leading to \eqref{eqN}, we obtain
 \begin{equation}\label{estGv}
 \int_0^1\frac{v'^2}{2}+\frac{1}{4\e^2}(1-v^2)^2+ \frac{\dd}{\e}f(\v(1))(1+O_\e(\sqrt{\e}) \leq  \frac{-\k^2 \e}{2\dd}(1+o_\e(1)).
 \end{equation} The study of the function $f_\e(x):=f(x)+\frac{1-\cos(x)}{4}O_\e(\sqrt{\e})$ on $[0,\pi]$ shows that  $f_\e(x)\geq \frac{-\k^2\e^2}{\d}(1+o_\e(1))$ for all $x$ in $[0,\pi]$.

 Therefore, \eqref{estGv} yields
 \begin{equation}\label{ameliorationsurv}
 \int_0^1\frac{v'^2}{2}+\frac{1}{4\e^2}(1-v^2)^2 \leq o_\e(\e).
 \end{equation}
 This is an improvement of \eqref{1eremajorationsurv}. We can then apply the method of the proof of Proposition \ref{convv1} to deduce from \eqref{ameliorationsurv} the $L^\infty$ estimate \eqref{93}.

 From \eqref{estGv} and the lower bound  $f(\v(1)) \geq \frac{-\k^2 \e^2}{2\d}(1+o_\e(1))$, we find $\frac{-\k^2 \e^2}{2\d}(1+o_\e(1))\leq f(\v(1)) \leq \frac{-\k^2 \e^2}{2\d}(1+o_\e(1))$. This implies, as in the proof of Proposition \ref{valuephi1}, that the first part of \eqref{15} holds. Moreover, going back to \eqref{eqmm}, and using the estimate for $\v (1)$, we see that the estimate holds when the integral is taken from 0 to $x$. So the same reasoning as before yields
 $$\frac {1-\cos |\v (x)|}4\leq \frac{\k^2\ep^2}{2\delta}(1+o_\ep(1)),$$ which implies the second part of \eqref{15}.

 This computation also yields that $G_{\e,\d,\k}(v,\v) \geq \frac{\dd}{\e}f(\v(1))(1+o_\ep(1))$, which, together with the upper bound \eqref{upperbG} and the estimate on $f(\v(1))$ yields \eqref{estGG}.

 \item[2)] The proof follows the same scheme as the preceding proof. This time we deduce from \eqref{upperbG} and \eqref{DL11} that
 \begin{equation}\label{supest}
\frac{1}{8}\int_0^1v^2\v'^2+\frac{\d}{8\e^2}\int_0^1v^4\sin^2\v-\frac{\k}{2}\int_0^1v^2\v'\leq  G_{\e,\d,\k}(v,\v) \leq 
f(\arcsin(2\l))
\frac{\dd}{\e}(1+o_\e(1)).
 \end{equation}
Inequality \eqref{1eremajorationsurv} yields that $\int_0^1(v^2-1)^2 \leq C\k^2 \e^2\leq C\l^2 \d$. Thus we can write
 \begin{equation}\label{ubvb}
\left| \frac{\k}{2}\int_0^1(v^2-1)\v' \right|\leq \frac{\k}{4}\left[\e^{2\a}\int_0^1\v'^2+C\d \e^{-2\a} \right]
 \end{equation}
with $\a$ to be chosen later. We also have from Proposition \ref{convv1} that \[\|v^2-1\|_{L^\infty((0,1))}\leq C\k \sqrt{\e}\leq C \l \sqrt{\frac{\d}{\e}} \text{ and } \|v^4-1\|_{L^\infty((0,1))}\leq  C \l \sqrt{\frac{\d}{\e}}. \]
We want to take $\a$ such that $\k \e^{2\a}=o_\e(\sqrt{\frac{\d}{\e}})$ and $\k \d \e^{-2\a}=o_\e(\frac{\dd}{\e})$. Since $\k=\frac{\l \dd}{\e}$, this leads to $ \e^{-1/2+2\a}=o_\e(1)$ and $\d \e^{-2\a}=o_\e(1) $. We use that $\d=o_\e(\e)$ and we see that the conditions are satisfied if  $\frac 1 4 <\a<\frac{1}{2}$. For such an $\a$ we obtain
\begin{eqnarray}\label{100}
\left( \frac{1}{8}\int_0^1 \v'^2+\frac{\d}{8\e^2}\int_0^1 \sin^2\v\right) (1+ O_\e(\sqrt{\frac{\d}{\e}}))
-\frac{\l \dd}{2\e}\v(1)  \nonumber \\
\phantom{aaaaa } \leq \left( \frac{1-\sqrt{1-4\l^2}}{4}-\frac{\l}{2}\arcsin(2\l) \right)\frac{\dd}{\e}(1+o_\e(1)).
\end{eqnarray} This implies that $\v(1)>0$.
 We let $N:=E\left(\frac{\v(1)}{\pi} \right)$ and we apply the Modica-Mortola technique to get
\begin{equation}
\frac{N}{2}\left(1-\l \pi \right)(1+O_\e(\sqrt{\frac{\d}{\e}}))+ f(\v(1)-N\pi)
\leq f( \arcsin(2\l) ) (1+o_\e(1)). \label{estcasb}
\end{equation}
with $f$ defined by \eqref{deff}. The study of the function $f$ shows that $f(\v(1)-N\pi) \geq f( \arcsin(2\l) )$ and thus
\begin{equation}\nonumber
\frac{N}{2}(1-\l \pi)(1+O_\e(\sqrt{\frac{\d}{\e}})) \leq o_\e(1).
\end{equation} Since $1-\tilde \kappa \pi >0$,
 this yields $N=0$ for $\e$ small enough. Now we use  the same previous inequalities but keeping track of the term $\int_0^1 \frac{v'^2}{8}+\frac{1}{4\e^2}(1-v^2)^2$ and we obtain
\begin{equation}
\int_0^1 \frac{v'^2}{8}+\frac{1}{4\e^2}(1-v^2)^2+ \frac{\dd}{\e}f(\v(1)) (1+O_\e(\sqrt{\frac{\d}{\e}}))
 \leq \frac{\dd}{\e} f( \arcsin(2\l) )(1+o_\e(1)). \label{supest2}
\end{equation}
But since $f(\v(1)) \geq f( \arcsin(2\l) )$, we obtain
\begin{equation}\nonumber
\int_0^1\frac{v'^2}{2}+\frac{1}{4\e^2}(1-v^2)^2 \leq o_\e({\frac{\dd}{\e}}).
\end{equation}
Now we can apply the method of the proof of Proposition \ref{convv1} to deduce the $L^\infty$ bound for $v$ and thus the first part of \eqref{94}. Since we have also found that $f(\v(1))=f( \arcsin(2\l) )(1+o_\e(1))$  this implies, like in Proposition \ref{valuephi1}, that $\v(1)=\arcsin(2\l)(1+o_\e(1))$.

We now come back to \eqref{100} and use the estimate for $\v(1)$ to find \begin{equation} \label{estphi}
  \frac{1}{8}\int_0^1 \v'^2+\frac{\d}{8\e^2}\int_0^1 \sin^2\v\leq \frac{\sqrt\delta}\ep\frac{(1-\sqrt{1-4\l^2})}{4} (1+o_\ep(1)).\end{equation} This upper bound also holds for the integral between 0 and any $x$ in $(0,1)$. Using the coaera formula, we thus find $$\frac{1-\cos |\v (x)|}4\leq \frac{1-\sqrt{1-4\l^2}}{4}(1+o_\e(1))$$ and this provides the required upper bound for $\v(x)$.

We recall that $G_{\e,\d,\k}(1,\v)=f(\v(1))\sqrt\delta/\ep(1+o_\e(1))$. Going back to \eqref{supest} and keeping track of the computations leading to \eqref{supest2}, we find $$G_{\e,\d,\k}(v,\v)=f(\v(1))\frac {\sqrt\delta}\ep(1+o_\e(1))$$ which is \eqref{estGG}.

\end{proof}

\begin{proposition}\label{utile}  If $\k=\frac{\l \dd}{\e}$, with $\l>\frac{1}{\pi}$ independent of $\e$, $\e^2=o_\e(\d)$ and $\d=o_\e(\e)$ then for a minimizer $(v_\e,\v_\e)$ of $G_{\e,\d,\k}$, we have
\begin{equation}\label{estGGG} G_{\e,\d,\k}(v_\e,\v_\e)=G_{\e,\d,\k}(1,\v_\e)(1+o_\e(1))
=\inf F_{\ep/\sqrt\delta,\k}(1+o_\e(1)).\end{equation}
If we assume that $\d=O_\e(\e^{3/2})$, then
\begin{equation}\label{eq:utile}
\int_0^1v'(x)^2\ dx +\frac{1}{4\e^2}\int_0^1(v^2(x)-1)^2\ dx=O_\e\left( \frac{\sqrt{\d}}{\e}\right).
\end{equation}
\end{proposition}
\begin{proof} Let $(v,\v$) be a minimizer of $G_{\e,\d,\k}$. Then \eqref{upperbG} holds.  Moreover,
\begin{equation}\label{estG3}
G_{\e,\d,\k}(v,\v) \geq \frac{1}{8} \int_0^1 v^2(\v'-2\k)^2+\frac{\d}{8\e^2}\int_0^1 v^4\sin^2 \v -\frac{\k^2}{2}.
\end{equation}
We use $\|v-1\|_{L^\infty((0,1))} \leq C \sqrt{\frac{\d}{\e}}$ (cf.\ Proposition \ref{convv1}) and proceed as in the previous proofs, approximating $v$ by 1 in \eqref{estG3}, to obtain
\[G_{\e,\d,\k}(v,\v)+\frac{\l^2 \d}{2\e^2}\geq\left(G_{\e,\d,\k}(1,\v)+\frac{\l^2 \d}{2\e^2}\right)\left(1+O_\e\left(\sqrt{\frac{\d}{\e}}\right)\right).\]

The upper bound \eqref{DLenergie2} shows that $\frac{\l^2 \d}{2\e^2}$ is of the same magnitude as the energy. Therefore,
\begin{equation}\label{fullenergycask>1/pi}
G_{\e,\d,\k}(v,\v)=G_{\e,\d,\k}(1,\v)\left(1+O_\e\left(\sqrt{\frac{\d}{\e}}\right)\right).
\end{equation}
From \eqref{DLenergie2}, we have the bounds $G_{\e,\d,\k}(v,\v) \leq \frac{-{\tilde{\a}}_0^2\d}{8\e^2}+O_\e\left(\frac{\sqrt{\d}}{\e}\right)$ and $G_{\e,\d,\k}(1,\v)\geq \frac{-{\tilde{\a}}_0^2 \d}{8\e^2}+O_\e\left(\frac{\sqrt{\d}}{\e}\right)$. This and \eqref{fullenergycask>1/pi} yield \eqref{estGGG}.

From \eqref{fullenergycask>1/pi}, we also deduce
\begin{equation}
\frac 12 \int_0^1 v'^2+\frac{1}{4\e^2}\int_0^1 (v^2-1)^2=O_\e\left(\frac{\sqrt{\d}}{\e}\right)+O_\e\left(\sqrt{\frac{\d}{\e}}\times \frac{\d}{\e^2} \right).
\end{equation}
Now using $\d=O_\e(\e^{3/2})$ we find \eqref{eq:utile}.\end{proof}

In the previous Proposition, the estimate $\delta =o(\ep)$ is not enough to have a sufficient bound for our later purposes on the energy for $v$. This is why we have added an extra hypothesis.

\subsection{Convergence of minimizers}
The study of the minimization of $F_{\e/\sqrt\d,\k}$ led to three different cases, for which we will prove convergence of $\v_\e$.
\begin{proposition}\label{approximationparF}
Let $\k>0$ and $\d>0$ be fixed, let $(v_\e,\v_\e)$ be a minimizer of $G_{\e,\d,\k}$ in $\I$.
If we set $ \Phi_\e(x):=\frac{\dd}{2\k \e} \v_\e(1-\frac{\e x}{\dd})$ then $\Phi_\e \rightarrow \Phi_0(x):=e^{-x}$ in $\mathcal{C}^1_{\text{loc}}(\R^+)$.
\end{proposition}
\begin{proof}
We recall from Proposition \ref{convv2} that, in this case we have $\v_\e(1)=\frac{2\k \e}{\sqrt{\d}}(1+o_\e(1))$. Next from \eqref{estGG}, we have $G_{\e,\d,\k}(1,\v_\e)=\frac{-\k^2 \e}{2\dd}(1+o_\e(1))$. These two facts imply that
\[ \int_0^1\frac{\v_\e'^2}{8}\leq \frac{\k^2 \e}{\dd}(1+o_\e(1)).\]
Now we set $w_\e(x)=v_\e(1-\frac{\e x}{\dd})$ and $\Phi_\e(x):=\frac{\dd}{2\k \e} \v_\e(1-\frac{\e x}{\dd})$. Both functions are defined in $[0,\frac{\dd}{\e}]$. With $\b=\frac{\e}{\dd}$  we find
\begin{equation}
\int_0^{\frac{1}{\b}} \Phi_\e'(x)^2dx
=\frac{1}{2\k^2 \b}\int_0^1 \v_\e'(y)^2\ dy\leq C
\end{equation}
for some constant $C$ independent of $\e, \k, \d$. Since $\Phi_\e(0)=\frac{\dd \v_\e(1)}{2\k \e}$ is bounded with respect to $\e$, we obtain that $\Phi_\e$ is bounded in $H^1_{loc}((0,\frac{1}{\b}))$ and hence there exists $\Phi_0$ in $H^1_{\text{loc}}(\R^+)$ such that, up to a subsequence, $\Phi_\e \rightharpoonup \Phi_0$ in $H^1_{\text{loc}}(\R^+)$ and  $\Phi_\e \rightarrow \Phi_0$ in $\mathcal{C}^0_{\text{loc}}(\R^+)$. Besides $(w_\ep,\Phi_\ep)$ satisfies the following equations
\begin{equation}\label{45}
\left\{
\begin{array}{lcll}
-w''+\frac{1}{\d}w(w^2-1)+\frac{\k^2\e^2}{\d}w \Phi'^2+\frac{1}{2}w^2\sin^2(2\k \b \Phi)+\frac{2\k^2\e^2}{\d} w\Phi'= \frac{\e^2}{\d}\lambda w, \\
-(w^2\Phi')'+w^4\cos (2\k \b \Phi) \frac{\sin(2\k \b \Phi)}{2\k \b}-2\k(w^2)'=0,
\end{array}
\right.
\end{equation}
\begin{equation}
\left\{
\begin{array}{lcll}
w'(0)=w'(\frac{1}{\b})=0, \\
\Phi(0)=\frac{\dd \v_\e(1)}{2\k \e}, \ \ \Phi'(0)=-1.
\end{array}
\right.
\end{equation}
Using that $w_\e \rightarrow 1$ in $\mathcal{C}^0_{\text{loc}}(\R^+)$ and $\Phi_\e \rightharpoonup \Phi_0$ in $H^1_{\text{loc}}(\R^+)$ we can pass to the limit in the sense of distributions in the second equation of \eqref{45} and find that $\Phi_0$ satisfies $\Phi_0''=\Phi_0$. Since the convergence is uniform on $[0,M]$ for every $M>0$ we find that $\Phi_0(0)=\lim_{\e\rightarrow 0} \frac{\dd \v_\e(1)}{2\k \e}=1$. We also have from the second equation of \eqref{45} that $\Phi_\e''$ is bounded in $L^2_{\text{loc}}((0,\frac{1}{\b}))$. Since $\Phi_\e$ is also bounded in $H^1((0,\frac{1}{\b}))$ we find that $\Phi'_\e$ converges in $\mathcal{C}^1_{\text{loc}}([0,+\infty))$ and thus  $\Phi_0'(0)=-1$. This yields $\Phi_0(x)=e^{-x}$. Since the limit is unique the entire sequence $\Phi_\e$ converges.
\end{proof}

\begin{proposition}
Let $\k=\frac{\l \dd}{\e}$, with $\l<\frac{1}{\pi}$ independent of $\e$, $\d=o_\e(\e)$ and $\e^2=o_\e(\d)$. Let $(v_\e,\v_\e)$ be a minimizer of $G_{\e,\d,\k}$ in $\mathcal{I}$. We set $\psi_\e(x)=\v_\e(1-\frac{\e x}{\dd})$ then  $\psi_\e \rightarrow\psi_0$ in $\mathcal{C}^1_{\text{loc}}(\R^+)$ with $\psi_0(x)=2\arctan \left[ \tan \left(\frac{\arcsin(2\l)}{2} \right)e^{-x}\right]$.
\end{proposition}

\begin{proof}
From Proposition \ref{convv2}, we have that $\v_\e(1)=\arcsin(2\l)(1+o_\e(1))$ and $G_{\e,\d,\k}(1,\v_\e)= \frac{\dd}{\e}\left( \frac{1-\sqrt{1-4\l^2}}{4}-\frac{\l}{2}\arcsin(2\l)\right) (1+o_\e(1)) $. We thus obtain
\[ \int_0^1 \v_\e'^2 \leq \frac{C \dd}{\e}, \]
for some $C>0$. Therefore,
\begin{eqnarray}
\int_0^{\frac{1}{\b}}\psi_\e'^2=\frac{\e}{\dd}\int_0^1 \v_\e'^2 \leq C. \nonumber
\end{eqnarray}
This proves that $\psi_\e$ is bounded in $H^1((0,\frac{1}{\b}))$ and thus there exists $\psi_0$ in $H^1_{\text{loc}}(\R^+)$ such that $\psi_\e \rightharpoonup \psi_0$ in $H^1_{\text{loc}}(\R^+)$, up to a subsequence. We have that $w_\e=v_\e(1-\frac{\e x}{\sqrt{\d}})$ and $\psi_\e$ satisfy
\begin{equation}\label{46}
\left\{
\begin{array}{lcll}
-w''+\frac{1}{\d}w(w^2-1)+\frac{1}{4}w\psi'^2+\l w \psi'=\lambda \frac{\e^2}{\d} w \text{ in } (0,\frac{1}{\b}), \\
-(w^2\psi')'+w^4\cos \psi \sin \psi -2\l (w^2)'=0 \text{ in } (0,\frac{1}{\b)}),
\end{array}
\right.
\end{equation}
\begin{equation}
\left\{
\begin{array}{lcll}
w'(0)=w'(1)=0,\\
\psi(0)= \v_\e(1), \ \ \ \psi'(0)=2\l.
\end{array}
\right.
\end{equation}
 We recall from Proposition \ref{convv2} that $w_\e\rightarrow 1$ in $\mathcal{C}^0_{\text{loc}}([0,+\infty))$.
We use that $w_\e \rightarrow 1$ in $\mathcal{C}^0$ and $\psi_\e \rightharpoonup \psi_0$ in $H^1_{\text{loc}}(\R^+)$ to pass to the limit in the sense of distributions in the second equation of \eqref{46} and we obtain
\begin{equation}
\psi_0''=\sin \psi_0 \cos \psi_0 \text{ in } \R^+.
\end{equation}
From the uniform convergence of $\psi_\e$ on every compact of $[0,+\infty)$, we find that $\psi_0(0)=
\lim_{\e\rightarrow 0} \v_\e(1)=\arcsin(2\l)$. Besides, we see from the second equation of \eqref{46}
that $\psi''_\e$ is bounded in $L^2_{\text{loc}}(\R^+)$. Since $\psi_\e$ is also bounded in
$H^1_{\text{loc}}(\R^+)$ (in particular $\psi_\e'$ is bounded in $L^2_{\text{loc}}(\R^+)$) we have that
$\psi_\e \rightarrow \psi_0$ in $\mathcal{C}^1_{\text{loc}}([0,+\infty))$. This means that $
\psi'_0(0)=2\l$. The unique solution with the desired boundary conditions is $\psi_0(x)= 2\arctan \left(\tan [ \frac{\arcsin(2\l)}{2}e^{-x}]\right) $. Since the limit is unique, the entire sequence $(\psi_\e)_\e$ converges.
\end{proof}

In the next case, we will show that $\v_\e$ goes beyond $N\pi$ with $N$ large, of order $\sqrt \d/\e$.
\begin{proposition}
Let $\k=\frac{\l \dd}{\e}$ with $\l>\frac{1}{\pi}$ independent of $\e$, $\e^2=o_\e(\d)$ and $\d=o_\e(\e)$. Let $(v_\e,\v_\e)$ be a minimizer of $G_{\e,\d,\k}$ in $\mathcal{I}$. We set $\psi_\e(x):=\frac{\e}{\dd}\v_\e(x)$, then there exists $\psi_0$ in $H^1((0,1))$ such that $\psi_\e \rightharpoonup \psi_0$ in $H^1((0,1))$ and $\psi_\e \rightarrow \psi_0$ in $\mathcal{C}^0([0,1])$ (up to a subsequence). Besides there exists $l>0$ such that $\lim_{\e\rightarrow 0} \frac{\e}{\dd}\v_\e(1)=l$.
\end{proposition}

\begin{proof}
We have $\psi_\e'(x)=\frac{\e}{\dd} \v_\e'(x)$ for $x$ in $[0,1]$. From  Proposition \ref{utile}, we also have $G_{\e,\d,\k}(1,\v_\e)=\frac{-{\tilde{\a}}_0^2\d}{{8}\e^2}(1+o_\e(1))$ with ${\tilde{\a}}_0$ defined by \eqref{defalpha0}. This means
\begin{equation}\nonumber
\int_0^1\frac{\v_\e'^2}{8}+\frac{\d}{8\e^2}\sin^2 \v_\e -\frac{\l \dd}{\e}\v_\e(1)=\frac{-{\tilde{\a}}_0^2\d}{{8}\e^2}(1+o_\e(1)).
\end{equation}
We use that
\[\int_0^1\frac{\v_\e'^2}{8}+\frac{\d}{8\e^2}\sin^2 \v_\e -\frac{\l \dd}{\e}\v_\e(1)= \int_0^1\frac{(\v_\e'-2\k)^2}{8}+\frac{\d}{8\e^2}\sin^2 \v_\e-\frac{\l^2 \d}{\e^2} \]
and $\k=\frac{\l \dd}{\e}$ to obtain that there exists $C>0$ independent of $\e$ such that
\begin{equation}\nonumber
\int_0^1(\psi_\e'-2\l)^2 \leq C.
\end{equation}
Since $\psi_\e(0)=0$ we have $\psi_\e$ bounded in $H^1((0,1))$. We deduce that there exists $\psi_0$ in $H^1((0,1))$ such that, up to a subsequence we have $\psi_\e \rightharpoonup \psi_0$ in $H^1((0,1))$ and $\psi_\e \rightarrow \psi_0$ in $\mathcal{C}^0([0,1])$. In particular there exists $l$ in $\R$ such that $\lim_{\e\rightarrow 0} \frac{\e}{\dd}\v_\e(1)=l$. We now show that $l>0$. We call $N=E\left(\frac{\v_\e(1)}{\pi}\right)$, by using \eqref{firstlowerbound} and the fact that $G_{\e,\d,\k}(1,\v_\e)=\frac{-{\tilde{\a}}_0^2\d}{\e^2}(1+o_\e(1))$ we find that $N\geq \frac{C_1 \dd}{\e}$ for some $C_1>0$ and hence $\lim_{\e\rightarrow 0} \frac{\e}{\dd}\v_\e(1)>0$.
\end{proof}

\begin{proposition}\label{bounddH1}
Let $\k=\frac{\l \dd}{\e}$, with $\l>\frac{1}{\pi}$ independent of $\e$, $\e^2=o_\e(\d)$ and $\d=O_\e(\e^{3/2})$. Let $(v_\e,\v_\e)$ be a minimizer of $G_{\e,\d,\k}$ in $\mathcal{J}$, we set $\tilde{\v}_\e(x):=\v_\e(\frac{\e x}{\dd})$. Then $\tilde{\v}_\e \rightarrow \v_0$ in $C^1_{\text{loc}}(\R^+)$ where $\v_0$ is the solution of \eqref{eqlimitk>1/pi}.
\end{proposition}

\begin{proof}
We set $w_\e(x):= v_\e(\frac{\e x}{\dd})$. Let $\psi_\e$ be a minimizer of $F_{\b,\k}$ in $\J$ where $\b=\e/\sqrt \d$. We  also set $\Phi_\e(x):=\psi_\e(\b x)$. For any large $R>0$, we want to prove that $(w_\e,\tilde{\v}_\e)$ is bounded in $H^1((0,R))$. In order to do that, we want to compare the energy of $(w_\e,\tilde{\v}_\e)$ with the energy of $(1, \Phi_\e)$ in $(0,R)$. In particular, we let :
\begin{eqnarray}\nonumber
g(w,\psi)&:=&\frac{w'^2}{2}+\frac{(1-w^2)^2}{4\d}+ \frac{1}{8}\left(w^2 \psi'^2+ w^4 \sin^2 \psi \right)-\frac{\l}{2}w^2\psi'\\
&=& \frac{w'^2}{2}+\frac{(1-w^2)^2}{4\d}+ \frac{1}{8}\left[w^2(\psi'-2\l)^2+w^4\sin^2\psi \right]-\frac{\l^2}{2}w^2 \nonumber
\end{eqnarray}
for $(w,\psi)$ in $H^1((0,\frac{1}{\b}))\times H^1((0,\frac{1}{\b}))$ and we want to show that there exists $C>0$ independent of $R$ such that
\begin{equation}\label{keyestimateR}
\int_0^R g(w_\e,\tilde{\v}_\e) \leq \int_0^R g(1,\Phi_\e) +C.
\end{equation}
By using Proposition \ref{utile} with the rescaled function $w_\e(x)=v_\e(\b x)$ we find
\begin{equation}
\frac{1}{2}\int_0^{\frac{1}{\b}}w_\e'^2+\frac{1}{4\d}\int_0^{\frac{1}{\b}} (w_\e^2-1)^2 =O_\e(1).
\end{equation}
Thus we have
\[ \int_{R+1}^{R+2} (w_\e^2-1)^2 \leq \int_0^\frac{1}{\b}(w_\e^2-1)^2 \leq C\delta . \]
Hence, there exists $1<x_0\leq 2$ such that
\begin{equation}\label{choixdex_0}
(w_\e(R+x_0)^2-1)^2 \leq C\delta.
\end{equation}
 As a test function, we take:
\begin{equation}
(u,\theta)=
\begin{cases}
(1, \Phi_\e) \ \text{ in } (0,R), \\
(w_\e,\tilde{\v}_\e-\pi E\left( \frac{\tilde{\v}_\e(R+x_0)-\Phi_\e(R)}{\pi } \right) ) \ \text{ in } (R+x_0, \frac{1}{\b}), \\
(u_1,\theta_1) \ \text{ in } (R,R+x_0),
\end{cases}
\end{equation}
with
\begin{eqnarray}
u_1 &=& \left(1-\frac{x-R}{x_0}\right)+ \frac{x-R}{x_0}w_\e(R+x_0) \nonumber \\
 \theta_1 &=& \left( 1-\frac{x-R}{x_0}\right)\Phi_\e(R)+ \frac{x-R}{x_0}\left( \tilde{\v}_\e(R+x_0)-\pi E\left( \frac{\tilde{\v}_\e(R+x_0)-\Phi_\e(R)}{\pi } \right) \right). \nonumber
\end{eqnarray}
We can then see, by using the minimizing property of $(w_\e,\tilde{\v}_\e)$ that
\begin{eqnarray}
\int_0^{R+x_0}g(w_\e,\tilde{\v}_\e) &\leq & \int_0^R g(1,\Phi_\e) \nonumber \\
& & +\frac{1}{8}\int_R^{R+x_0} u_1^2 \theta_1'^2+u_1^4\sin^2\theta_1-\frac{\l}{2}\int_R^{R+x_0}u_1^2\theta_1' \nonumber \\
& & +\int_R^{R+x_0} \frac{u_1'^2}{2}+\frac{(u_1^2-1)^2}{4\d}, \nonumber
\end{eqnarray}
where we have used that
\begin{equation}\label{reprise}
\int_{R+x_0}^\frac{1}{\b}g(w,\psi+L\pi)=\int_{R+x_0}^\frac{1}{\b}g(w,\psi)
\end{equation}
for every $L$ in $\mathbb{Z}$.
Since we have that
$u'(x)=\frac{w_\e(R+x_0)-1}{x_0}$ in $(R,R+x_0)$ and $(w_\e(R+x_0)^2-1)^2\leq C\d$ (from \eqref{choixdex_0}) we can see that
\[ \int_R^{R+x_0} \frac{u'^2}{2}+\frac{(u^2-1)^2}{4\d} \leq \frac{C \d}{x_0}+C\leq C \]
where $C$ is independent of $R$. We also have that
\[ \theta'(x)=\frac{1}{x_0}\left( \tilde{\v}_\e(R+x_0)-\Phi_\e(R)- \pi E\left( \frac{\tilde{\v}_\e(R+x_0)-\Phi_\e(R)}{\pi } \right)\right). \]
Hence we find, by using that $x_0>1$, that $ |\theta'(x)|\leq \pi$. Since $\|w_\e-1\|_{L^\infty} \leq C \sqrt{\frac{\d}{\e}}$ we get
\begin{equation}\nonumber
\frac{1}{8}\int_R^{R+x_0} u^2 \theta'^2+u^4\sin^2\theta-\frac{\l}{2}\int_R^{R+x_0}u^2\theta' \leq C
\end{equation}
with $C$ independent of $R$. This yields
\begin{equation}\nonumber
\int_0^{R+x_0} g(w_\e,\tilde{\v}_\e) \leq \int_0^R g(1,\Phi_\e) +C.
\end{equation}
By using the expression of $g(w,\psi)$ we can see that
\begin{eqnarray}\nonumber
\int_0^{R} g(w_\e, \tilde{\v}_\e) &\leq& \int_0^{R+x_0} g(w_\e, \tilde{\v}_\e)+\frac{\l^2}{2}\int_R^{R+x_0}w^2 \\
&\leq & \int_0^R g(1,\Phi_\e) +\frac{\l^2}{2}.
\end{eqnarray}
Thus we obtain \eqref{keyestimateR}.
Since we know that $|\int_0^R g(1,\Phi_\e)|$ is bounded uniformly in $\e$, we  deduce that $(w_\e,\tilde{\v}_\e)$ is bounded in $H^1((0,R))$ for every $R>0$. We can thus find $\v_0$ in $H^1_{\text{loc}}(\R^+)$ and extract a subsequence, still denoted by $\e$ such that $\tilde{\v}_\e \rightharpoonup \v_0$ in $H^1_{\text{loc}}(\R^+)$ and $\tilde{\v}_\e \rightarrow \v_0$ in $C^0_{\text{loc}}(\R^+)$. We also have that $(w_\e,\tilde{\v}_\e)$ satisfy
\begin{equation}\label{778}
\left\{
\begin{array}{lcll}
-w''+\frac{1}{\d}w(w^2-1)+\frac{1}{4}w\tilde{\v}'^2-\l w \tilde{\v}'=\lambda \frac{\e^2}{\d} w \text{  in } (0,\frac{1}{\b}), \\
-(w^2\tilde{\v}')'+w^4\cos \tilde{\v} \sin \tilde{\v} +2\l (w^2)'=0 \text{  in } (0,\frac{1}{\b}),
\end{array}
\right.
\end{equation}
\begin{equation}
\left\{
\begin{array}{lcll}
w'(0)=w'(\frac{1}{\b})=0,\\
\tilde{\v}(0)= 0, \ \ \ \tilde{\v}'(\frac{1}{\b})=2\l.
\end{array}
\right.
\end{equation}
By using that $\tilde{\v}_\e \rightarrow \tilde{\v}_0$ in $\mathcal{C}^0_{\text{loc}}([0,+\infty))$, $w_\e \rightarrow 1$ in $\mathcal{C}^0_{\text{loc}}([0,+\infty))$ and $\tilde{\v}_\e {\rightharpoonup} \tilde{\v}_0$ in $H^1_{\text{loc}}(\R^+)$ we can pass to the limit in the sense of distributions in the second equation of \eqref{778} and we find that $\tilde{\v}_0$ satisfies
\begin{equation}\nonumber
\left\{
\begin{array}{lcll}
\tilde{\v}_0''&=&\sin \tilde{\v}_0 \cos \tilde{\v}_0\text{ in } \R^+ ,\\
\tilde{\v}_0(0)&=&0.
\end{array}
\right.
\end{equation}
Furthermore the second equation of \eqref{778} also provides us with $\tilde{\v}_\e''$ bounded in $L^2_{\text{loc}}(\R^+)$ because $w_\e'$ is bounded in $L^2_{\text{loc}}(\R^+)$. Then, since $\tilde{\v}_\e$ is already bounded in $H^1_{\text{loc}}(\R^+)$ we have that $\tilde{\v}_\e \rightarrow \tilde{\v}_0$ in $\mathcal{C}^1_{\text{loc}}([0,+\infty))$. In particular we have that
\begin{equation}\nonumber
\tilde{\v}_0'(0)=\omega_0
\end{equation}
for some $\omega_0 >0$ such that $\lim_{\e\rightarrow 0} \tilde{\v}_\e'(0)=\omega_0$. The rest of the proof is devoted to showing  $\omega_0={\tilde{\a}}_0$, where ${\tilde{\a}}_0$ is defined by \eqref{defalpha0}.

We  pass to the limit $\e \rightarrow 0$ in  inequality \eqref{keyestimateR}, by using the convergence $C^1_{\text{loc}}(\R^+)$ of $ \tilde{\v}_\e$ and $\Phi_\e$ and the weak convergence in $H^1_{\text{loc}}(\R^+)$ of $w_\e$ to obtain
\begin{eqnarray}\label{prepassage}
\frac{1}{R}\int_0^R g(1, \tilde{\v}_0) \leq \frac{1}{R}\int_0^R g(1,\Phi_0) +\frac{C}{R},
\end{eqnarray}
We now use that $ \tilde{\v}_0$ and $\Phi_0$ are periodic in the sense that there exist $T',T$ such that
$ \tilde{\v}_0(x+T')=\pi+\psi_0(x)$ in $[0,\frac{1}{\b}-T']$ and $\Phi_0(x+T)=\pi+\Phi_0(x)$ in $[0,\frac{1}{\b}-T]$ and we take the limit $R$ goes to infinity in \eqref{prepassage}  to obtain
\begin{eqnarray}\label{prelimit}
\frac{1}{T'}\int_0^{T'} g(1, \tilde{\v}_0) \leq \frac{1}{T}\int_0^T g(1,\Phi_0).
\end{eqnarray}
 We claim that $\frac{1}{T'}\int_0^{T'} g(1, \tilde{\v}_0)=h(\omega_0)$ and $\frac{1}{T}\int_0^T g(1,\Phi_0)=h({\tilde{\a}}_0)$ where $h$ is defined by \eqref{defh2}. Let us show the first equality, the second being similar. By definition $T'$ is the first point such that $\tilde{\v}_0(T')=\pi$. By using the equation of $\tilde{\v}_0$, and in particular the fact that $\tilde{\v}_0'^2=\sin^2\tilde{\v}_0+\omega_0^2$, we have that
\begin{equation}
T'=\int_0^\pi \frac{dy}{\sqrt{\sin^2y+\omega_0^2}}.
\end{equation}
Besides we have
\begin{eqnarray}
\int_0^{T'} g(1,\tilde{\v}_0) & = & \frac{1}{8}\int_0^{T'} \tilde{\v}_0'(y)^2+\sin^2 \tilde{\v}_0(y)dy -\frac{\l \pi}{2} \nonumber \\
&= & \frac{1}{8}\int_0^{T'} \tilde{\v}_0'(y)^2+\sin^2 \tilde{\v}_0(y) +\omega_0^2dy -\frac{\l \pi}{2}-\frac{\omega_0^2 T'}{8} \nonumber \\
& =& \frac{1}{4} \int_0^{T'} |\tilde{\v}_0'(y)| \sqrt{\sin^2\tilde{\v}_0(y)+\omega_0^2}dy -\frac{\l \pi}{2}-\frac{\omega_0^2 T'}{8} \nonumber \\
&=& \frac{1}{4}\int_0^\pi \sqrt{\sin^2y+\omega_0^2}-\frac{\l \pi}{2}-\frac{\omega_0^2 T'}{8} \nonumber
\end{eqnarray}
where we have used that  $\tilde{\v}_0'=\sqrt{\sin^2\tilde{\v}_0+\omega_0^2}$. We thus have
\begin{eqnarray}
\frac{1}{T'}\int_0^{T'} g(1,\tilde{\v}_0)&=& \frac{\int_0^\pi \sqrt{\sin^2y+\omega_0^2dy}-2\l \pi}{4\int_0^\pi\frac{dy}{\sqrt{\sin^2y+\omega_0^2}}}-\frac{\omega_0^2}{8} \nonumber \\
&=&\frac{\int_0^{\pi/2} \sqrt{\sin^2y+\omega_0^2dy}-\l \pi}{4\int_0^{\pi/2}\frac{dy}{\sqrt{\sin^2y+\omega_0^2}}}-\frac{\omega_0^2}{8} \nonumber \\
&=&h(\omega_0) \nonumber.
\end{eqnarray}
Therefore \eqref{prelimit} implies
$h(\omega_0) \leq h({\tilde{\a}}_0)$. But ${\tilde{\a}}_0$ is the unique minimizer of $h$, as proved in the proof of Proposition \ref{propcask>1surpi}. Thus we have $\omega_0={\tilde{\a}}_0$ and $\v_0$ is the solution of \eqref{eqlimitk>1/pi}.
\end{proof}
Theorem \ref{minimizerspbcomplet} follows from the Propositions of this section.

\hfill

\noindent
{\bf Acknowledgements:} The first author would like to thank Felix Otto for discussing Modica Mortola techniques. The visit of the second author to  the departement of mathematics of the Universit\'e Versailles-Saint-Quentin  was supported by a public grant as part of the Investissement d'avenir project, reference ANR-11-LABX-0056-LMH, LabEx LMH.  The second author has also been supported by the Millennium Nucleus Center for Analysis of PDE NC130017 of the Chilean Ministry of Economy.

\bibliographystyle{plain}
\bibliography{biblio}
\end{document}